\pdfoutput=1
\RequirePackage{ifpdf}
\ifpdf 
\documentclass[pdftex]{sigma}
\else
\documentclass{sigma}
\fi

\numberwithin{equation}{section}

\usepackage{bm}

\usepackage[all]{xy}

\newtheorem{Theorem}{Theorem}[section]
\newtheorem{Proposition}[Theorem]{Proposition}
\newtheorem{Corollary}[Theorem]{Corollary}
\newtheorem{Lemma}[Theorem]{Lemma}

\theoremstyle{definition}
\newtheorem{Remark}[Theorem]{Remark}
\newtheorem{Definition}[Theorem]{Definition}

\def\cc{\mathbf{c}}

\def\ff{\mathbf{f}}
\def\gg{\mathbf{g}}

\def\xx{\mathbf{x}}
\def\yy{\mathbf{y}}
\def\TT{\mathbb{T}}
\def\PP{\mathbb{P}}
\def\ZZ{\mathbb{Z}}
\def\QQ{\mathbb{Q}}

\def\Fcal{\mathcal{F}}
\def\Xcal{\mathcal{X}}
\def\Ycal{\mathcal{Y}}

\def\QQsf{\mathbb{Q}_{\text{sf}}}
\def\trop{\text{Trop}}

\begin{document}
\allowdisplaybreaks

\newcommand{\arXivNumber}{1808.02156}

\renewcommand{\PaperNumber}{040}

\FirstPageHeading

\ShortArticleName{Duality between Final-Seed and Initial-Seed Mutations in Cluster Algebras}

\ArticleName{Duality between Final-Seed and Initial-Seed \\ Mutations in Cluster Algebras}

\Author{Shogo FUJIWARA and Yasuaki GYODA }

\AuthorNameForHeading{S.~Fujiwara and Y.~Gyoda}
\Address{Graduate School of Mathematics, Nagoya University, Chikusa-ku, Nagoya, 464-8602 Japan}
\Email{\href{mailto:m17036g@math.nagoya-u.ac.jp}{m17036g@math.nagoya-u.ac.jp}, \href{mailto:m17009g@math.nagoya-u.ac.jp}{m17009g@math.nagoya-u.ac.jp}}

\ArticleDates{Received October 16, 2018, in final form May 10, 2019; Published online May 15, 2019}

\Abstract{We study the duality between the mutations and the initial-seed mutations in cluster algebras, where the initial-seed mutations are the transformations of rational expressions of cluster variables in terms of the initial cluster under the change of the initial cluster. In particular, we define the maximal degree matrices of the $F$-polynomials called the $F$-matrices and show that the $F$-matrices have the self-duality which is analogous to the duality between the $C$- and $G$-matrices.}


\Classification{13F60}

\section{Introduction}
\emph{Cluster algebras} are a class of commutative algebras introduced by \cite{fzi}, which are generated by some distinguished elements called the \emph{cluster variables}.
The cluster variables are given by applying the \emph{mutations} repeatedly starting from the initial cluster variables. Through the description of mutations as transformations on quivers or triangulations, the cluster algebras are applied in various areas of mathematics such that representation theory of quivers \cite{bmrrt}, hyperbolic geometry \cite{fg09, fst}, etc.

Let us briefly recall the notion of mutations, which is essential in cluster algebra theory. For simplicity, we consider here a cluster pattern \emph{without coefficients} \cite{fzi}. Let $\TT_n$ be the $n$-regular tree. Then, for each vertex $t\in\TT_n$, cluster variables $\xx_t=(x_{1;t}, \dots, x_{n;t})$ are attached such that for any adjacent vertices $t$ and $t'$ of $\TT_n$, the variables $\xx_t$ and $\xx'_t$ are related by a rational transformation called a \emph{mutation}. Let $t_0\in\TT_n$ be a given point in $\TT_n$ called the \emph{initial point}. Then, repeating mutations starting from the initial variables, we have the expression
\begin{gather*}
	x_{i;t}=\Xcal^{t_0}_{i;t}(\xx_{t_0}),
\end{gather*}
where \looseness=-1 the rational function $\Xcal^{t_0}_{i;t}$ depends on $t$ and $t_0$. For a vertex $t'$ which is adjacent to $t$ in $\TT_n$, the mutation transforms the function $\Xcal^{t_0}_{i;t}$ to the function $\Xcal^{t_0}_{i;t'}$. On the other hand, for a~ver\-tex~$t_1$ which is adjacent to $t_0$ in $\TT_n$, the mutation also transforms the function~$\Xcal^{t_0}_{i;t}$ to the func\-tion~$\Xcal^{t_1}_{i;t}$. We call this transformation the \emph{initial-seed mutations}. (In contrast to the initial-seed mutations, we call the ordinary mutations also as the \emph{final-seed mutations} in this paper.)

Thanks to the \emph{separation formulas} of \cite{fziv}, the cluster variables and the coefficients are described by the \emph{$C$-matrices}, the \emph{$G$-matrices}, and the \emph{$F$-polynomials}, where the $C$- and $G$-matrices are the ``\emph{tropical part}'', while the $F$-polynomials are the ``\emph{nontropical part}''. The final-seed and initial-seed mutations of the cluster variables and the coefficients reduce to the ones for the $C$- and $G$-matrices and the $F$-polynomials~\cite{fziv}.
In \cite{nz}, the duality between the $C$- and $G$-matrices under the final-seed and initial-seed mutations was established (under the assumption of the sign-coherence of the $C$-matrices, which is established in \cite{ghkk}). In~\cite{rs}, an analogous duality was established also for the \emph{denominator vectors} for some classes of cluster algebras.
 In this paper, we treat the nontropical part of the seeds, namely the $F$-polynomials. Since the $F$-polynomials are rather complicated, we consider the ``degree matrices'' of $F$-polynomials (\emph{F-matrices}). Under applying the sign-coherence of the $C$-matrices, a column vectors of the $F$-matrices correspond to the exponent vector of the unique monomial with maximal degree of a $F$-polynomials.
 Then, we establish the duality of those matrices under the final-seed and initial-seed mutations, which is parallel to the ones in \cite{nz, rs}. This is the main result of the paper. Also, in a process of proving the main result, we have a nice relation between the $C$-, $G$-, $F$-matrices under the exchange of the initial matrix $B$ and $-B$.

We also exhibit the final-seed and initial-seed mutation formulas of the $C$- and $G$-matrices systematically, some of which are new in the literature. We stress that the duality of the final-seed and initial-seed mutations becomes manifest only after applying the sign-coherence of the $C$-matrices~\cite{ghkk}.

The organisation of the paper is as follows:
In Section~\ref{section2}, we start with recalling the definitions of the seed mutations and the cluster patterns according to~\cite{fziv} and give the final-seed mutations.
 Also, using the sign-coherence of the $C$-matrices, we obtain a reduced form of the final-seed mutations. The formulas involving the $F$-matrices are new results.
In Section~\ref{section3}, the initial-seed mutation formulas are given using the \emph{$H$-matrices} of~\cite{fziv}. By sign-coherence of the $C$-matrices and the duality of the $C$- and $G$-matrices of \cite{nz}, we prove the conjecture \cite[Conjecture~6.10]{fziv}, where the $H$-matrices are expressed by the $G$-matrices. Using this result, we obtain the duality of the $F$-matrices.
In Section~\ref{section4}, we consider the principal extension of the exchange matrices and present some properties. We show that these properties give alternative derivations of some equalities in previous sections.

\section{Final-seed mutations}\label{section2}
\subsection{Seed mutations and cluster patterns}\label{section2.1}
We start with recalling the definitions of the seed mutations and the cluster patterns according to \cite{fziv}. See \cite{fziv} for more information. A \emph{semifield} $\mathbb P$ is an abelian multiplicative group equipped with an addition $\oplus$ which is distributive over the multiplication. We particularly make use of the following two semifields.

Let $\mathbb Q_{\text{sf}}(u_1,\dots,u_{\ell})$ be a set of rational functions in $u_1,\dots,u_{\ell}$ which have subtraction-free expressions. Then, $\mathbb Q_{\text{sf}}(u_1,u_2,\dots,u_{\ell})$ is a semifield by the usual multiplication and addition. We call it the \emph{universal semifield} of $u_1,\dots,u_{\ell}$ \cite[Definition~2.1]{fziv}.

Let Trop$(u_1,\dots, u_\ell)$ be the abelian multiplicative group freely generated by the elements $u_1,\dots,u_\ell$. Then, $\text{Trop}(u_1,u_2,\dots,u_{\ell})$ is a semifield by the following addition:
\begin{gather*}
\prod_{j=1}^\ell u_j^{a_j} \oplus \prod_{j=1}^{\ell} u_j^{b_j}=\prod_{j=1}^{\ell} u_j^{\min(a_j,b_j)}.
\end{gather*}
We call it the \emph{tropical semifield} of $u_1,\dots,u_\ell$ \cite[Definition~2.2]{fziv}. We note that the universal semifield $\mathbb Q_{\text{sf}}(u_1,\dots,u_{\ell})$ has the following universality property.
For any semifield $\PP$ and $p_1, \dots, p_{\ell}\in\PP$, there exists an unique semifield homomorphism $\pi$ such that
\begin{align*}
	\pi\colon \ \QQsf(y_1, \dots, y_{\ell}) &\longrightarrow \PP,\nonumber\\
	y_i &\longmapsto p_i. 
\end{align*}
For $F(y_1,\dots,y_\ell ) \in \QQsf(y_1, \dots, y_{\ell})$, we denote{\samepage
\begin{gather*}
	F|_{\PP}(p_1, \dots, p_{\ell}):=\pi(F(y_1, \dots, y_\ell))
\end{gather*}
and we call it the \emph{evaluation} of $F$ at $p_1, \dots, p_{\ell}$.}

We fix a positive integer $n$ and a semifield $\PP$. Let $\mathbb{ZP}$ be the group ring of $\mathbb{P}$ as a multiplicative group. Since $\mathbb{ZP}$ is a domain \cite[Section~5]{fzi}, its total quotient ring is a field $\mathbb{Q}(\mathbb P)$. Let $\mathcal{F}$ be the field of the rational functions in $n$ indeterminates with coefficients in $\mathbb{Q}(\mathbb P)$.

A \emph{labeled seed with coefficients in $\PP$} is a triplet $(\mathbf{x}, \mathbf{y}, B)$, where
\begin{itemize}\itemsep=0pt
\item $\mathbf{x}=(x_1, \dots, x_n)$ is an $n$-tuple of elements of~$\mathcal F$ forming a free generating set of~$\mathcal F$.
\item $\mathbf{y}=(y_1, \dots, y_n)$ is an $n$-tuple of elements of~$\mathbb{P}$.
\item $B=(b_{ij})$ is an $n \times n$ integer matrix which is \emph{skew-symmetrizable}, that is, there exists a positive diagonal matrix $D$ such that $DB$ is skew-symmetric. We call $D$ a \emph{skew-symmetrizer} of $B$.
\end{itemize}
We refer to $x_i$, $y_i$ and $B$ as the \emph{cluster variables}, the \emph{coefficients} and the \emph{exchange matrix}, respectively.

Throughout the paper, for an integer $b$, we use the notation $[b]_+=\max(b,0)$. We note that
\begin{gather}\label{eq:b--b}
b=[b]_+-[-b]_+.
\end{gather}
Let $(\mathbf{x}, \mathbf{y}, B)$ be a labeled seed with coefficients in $\PP$, and let $k \in\{1,\dots, n\}$. The \emph{seed mutation~$\mu_k$ in direction $k$} transforms $(\mathbf{x}, \mathbf{y}, B)$ into another labeled seed $\mu_k(\mathbf{x}, \mathbf{y}, B)=(\mathbf{x'}, \mathbf{y'}, B')$ defined as follows:
\begin{itemize}\itemsep=0pt
\item The entries of $B'=(b'_{ij})$ are given
\begin{gather} \label{eq:matrix-mutation}
b'_{ij}=\begin{cases}-b_{ij} &\text{if $i=k$ or $j=k$,} \\
b_{ij}+\left[ b_{ik}\right] _{+}b_{kj}+b_{ik}\left[ -b_{kj}\right]_+ &\text{otherwise.}
\end{cases}
\end{gather}
\item The coefficients $\mathbf{y'}=(y'_1, \dots, y'_n)$ are given
\begin{gather}\label{eq:y-mutation}
y'_j=
\begin{cases}
y_{k}^{-1} &\text{if $j=k$,} \\
y_j y_k^{[b_{kj}]_+}(y_k \oplus 1)^{-b_{kj}} &\text{otherwise.}
\end{cases}
\end{gather}
\item The cluster variables $\mathbf{x'}=(x'_1, \dots, x'_n)$ are given
\begin{gather}\label{eq:x-mutation}
x'_j=\begin{cases}\dfrac{y_k\mathop{\prod}\limits_{i=1}^{n} x_i^{[b_{ik}]_+}+\mathop{\prod}\limits_{i=1}^{n} x_i^{[-b_{ik}]_+}}{(y_k\oplus 1)x_k} &\text{if $j=k$,}\\
x_j &\text{otherwise.}
\end{cases}
\end{gather}
\end{itemize}

We remark that \eqref{eq:y-mutation} can be also expressed as follows:
\begin{gather}\label{eq:y-mutation2}
y'_j=
\begin{cases}
y_{k}^{-1} &\text{if $j=k$,} \\
y_j y_k^{[-b_{kj}]_+}\big({y_k}^{-1} \oplus 1\big)^{-b_{kj}} &\text{otherwise.}
\end{cases}
\end{gather}

Let $\mathbb{T}_n$ be the \emph{$n$-regular tree} whose edges are labeled by the numbers $1, \dots, n$ such that the~$n$ edges emanating from each vertex have different labels. We write
\begin{xy}(0,1)*+{t}="A",(10,1)*+{t'}="B",\ar@{-}^k"A";"B" \end{xy}
to indicate that vertices $t,t'\in \mathbb{T}_n$ are joined by an edge labeled by $k$. We fix an arbitrary vertex $t_0\in \TT_n$, which is called the \emph{initial vertex}.

A \emph{cluster pattern with coefficients in $\PP$} is an assignment of a labeled seed $\Sigma_t=(\mathbf{x}_t, \mathbf{y}_t,B_t)$ with coefficients in $\PP$ to every vertex $t\in \mathbb{T}_n$ such that the seeds $\Sigma_t$ and $\Sigma_{t'}$ assigned to the endpoints of any edge
\begin{xy}(0,1)*+{t}="A",(10,1)*+{t'}="B",\ar@{-}^k"A";"B" \end{xy}
are obtained from each other by the seed mutation in direction~$k$. The elements of $\Sigma_t$ are denoted as follows:
\begin{gather} \label{den:seed_at_t}
\mathbf{x}_t=(x_{1;t},\dots,x_{n;t}),\qquad \mathbf{y}_t=(y_{1;t},\dots,y_{n;t}),\qquad B_t=(b_{ij;t}).
\end{gather}
In particular, at $t_0$, we denote
\begin{gather} \label{initialseed}
\mathbf{x}=\mathbf{x}_{t_0}=(x_1,\dots,x_n),\qquad \mathbf{y}=\mathbf{y}_{t_0}=(y_1,\dots,y_n),\qquad B=B_{t_0}=(b_{ij}).
\end{gather}
We say that $n$ is \emph{rank} and the $\PP$ is \emph{coefficient semifield} of the cluster patterns.

If $\mathbb{P}=\text{Trop}(y_1,\dots,y_n)$ and $\mathbf{y}_{t_0}=(y_1,\dots,y_n)$, we say that a cluster pattern $\{(\mathbf{x}_t, \mathbf{y}_t,B_t)\}_{t\in\TT_n}$ has \emph{principal coefficients} at the initial vertex $t_0$.

\subsection[Final-seed mutations without sign-coherence of $C$-matrices]{Final-seed mutations without sign-coherence of $\boldsymbol{C}$-matrices}\label{section2.2}

In this subsection, we will define the $C$-matrices, the $G$-matrices and the $F$-polynomials fol\-lo\-wing~\cite{fziv}.
We also introduce the $F$-matrices which are new.

Throughout this paper, we use the following notations \cite{nz}. Let $J_{\ell}$ denote the $n\times n$ diagonal matrix obtained from the identity matrix $I_n$ by replacing the $(\ell, \ell)$ entry with $-1$.
For a $n\times n$ matrix $B=(b_{ij})$, let $[B]_+$ be the matrix obtained from $B$ by replacing every entry~$b_{ij}$ with~$[b_{ij}]_+$.
Also, let $B^{k\bullet}$ be the matrix obtained from $B$ by replacing all entries outside of the $k$th row with zeros. Similarly, let $B^{\bullet k}$ be the matrix replacing all entries outside of the $k$th column.
Note that the maps $B\mapsto [B]_+$ and $B\mapsto B^{k\bullet}$ commute with each other, and the same is true of $B\mapsto [B]_+$ and $B\mapsto B^{\bullet k}$, so that the notations $[B]^{k\bullet}_+$ and $[B]^{\bullet k}_+$ make sense. Also, we have $AB^{\bullet k}=(AB)^{\bullet k}$ and $A^{k\bullet}B=(AB)^{k\bullet}$.

First, we recall the recursive definitions of the $C$-matrices, the $G$-matrices and the $F$-polynomials, following~\cite{fziv}. For any initial exchange matrix $B$ at $t_0$, let $\{B_t\}_{t\in\TT_n}$ be the family of the exchange matrices in~\eqref{den:seed_at_t} with~$B_{t_0}=B$.

\begin{Definition}[{\cite[equation~(5.9), Proposition~6.6]{fziv}}]
Let $B$ be any initial exchange matrix at~$t_0$. Then, the families of $n\times n$ integer matrices $\big\{C_t^{B;t_0}\big\}_{t\in\TT_n}$ and $\big\{G_t^{B;t_0}\big\}_{t\in \TT_n}$ are recursively defined as follows:
\begin{itemize}\itemsep=0pt
\item[(i)] We set the initial condition,
\begin{gather*}
C^{B; t_0}_{t_0}=I_n,
\end{gather*}
and for any edge \begin{xy}(0,1)*+{t}="A",(10,1)*+{t'}="B",\ar@{-}^{\ell}"A";"B" \end{xy} in $\TT_n$ {and $\varepsilon\in\{\pm1\}$}, we set the recurrence relation,
\begin{gather}\label{eq:c-frontmutation}
C^{B; t_0}_{t'}=C^{B; t_0}_t\big(J_\ell+[\varepsilon B_t]^{\ell\bullet}_+\big)+\big[{-}\varepsilon C^{B; t_0}_t\big]^{\bullet \ell}_+ B_t.
\end{gather}
(In \eqref{eq:c-frontmutation}, it does not make any difference whichever we choose, $\varepsilon=1$ or $\varepsilon=-1$. See Remark~\ref{chooseep}.)

\item[(ii)] We set the initial condition,
\begin{gather*}
G^{B; t_0}_{t_0}=I_n,
\end{gather*}
and for any edge \begin{xy}(0,1)*+{t}="A",(10,1)*+{t'}="B",\ar@{-}^{\ell}"A";"B" \end{xy} in $\TT_n$ and $\varepsilon\in\{\pm1\}$, we set the recurrence relation,
\begin{gather}\label{eq:g-frontmutation}
 G^{B; t_0}_{t'}=G^{B; t_0}_t\big(J_\ell+[\varepsilon B_t]^{\bullet \ell}_+\big)-B\big[\varepsilon C^{B; t_0}_t\big]^{\bullet \ell}_+.
\end{gather}
\end{itemize}
The matrices $C^{B; t_0}_t$ and $G^{B; t_0}_t$ are called the \emph{$C$-matrix} and the \emph{$G$-matrix} at~$t$.
\end{Definition}
\begin{Remark}\label{chooseep} Because of \eqref{eq:b--b}, the right hand side of~\eqref{eq:c-frontmutation} does not depend on $\varepsilon$.
Meanwhile, the right hand side of~\eqref{eq:g-frontmutation} does not depend on $\varepsilon$ due to the following equality \cite[equation~(6.14)]{fziv}:
\begin{gather} \label{GB=BC}
G^{B; t_0}_tB_t=BC^{B; t_0}_t.
\end{gather}
We have two different expressions for the recursion \eqref{eq:c-frontmutation} and \eqref{eq:g-frontmutation} because they will be useful in different situations.
\end{Remark}

\begin{Definition} [{\cite[Proposition 5.1]{fziv}}] \label{Frec}
Let $B$ be any initial exchange matrix at $t_0$ and $\mathbf{y}=(y_1,\dots,y_n)$ be formal variables.
Then, the family of polynomials $F^{B; t_0}_{j; t}(\mathbf{y})\in \ZZ[y_1, \dots, y_n]$ indexed by $j\in \{1, \dots, n\}$ and $t\in \TT_n$ are recursively defined as follows: We set the initial condition,
\begin{gather}
F^{B; t_0}_{j; t_0}(\mathbf{y}) = 1, \qquad j = 1, \dots, n,\label{eq:f-initialcondition}
\end{gather}
and for any edge \begin{xy}(0,1)*+{t}="A",(10,1)*+{t'}="B",\ar@{-}^{\ell}"A";"B" \end{xy} in $\TT_n$, we set the recurrence relation,
\begin{gather}\label{eq:f-frontmutation}
	F^{B; t_0}_{j; t'}(\mathbf{y})=\begin{cases} \displaystyle
F^{B; t_0}_{\ell; t}(\mathbf{y})^{-1}\left(\prod_{i = 1}^n y_i^{[c^{B; t_0}_{i\ell; t}]_+} \prod_{i = 1}^n F^{B;t_0}_{i;t}(\mathbf{y})^{[b_{i\ell; t}]_+}\right. \\
\qquad\left.\displaystyle{} + \prod_{i = 1}^n y_i^{[-c^{B; t_0}_{i\ell; t}]_+} \prod_{i = 1}^n F^{B;t_0}_{i;t}(\mathbf{y})^{[-b_{i\ell; t}]_+}\right) \quad &\text{if} \ j = \ell,\\
F^{B; t_0}_{j, t}(\mathbf{y}) & \text{if} \ j \neq \ell,
\end{cases}
\end{gather}
where $c^{B; t_0}_{i\ell; t}$ is the $(i, \ell)$ entry of $C^{B; t_0}_t$ and $b_{i\ell; t}$ is the $(i, \ell)$ entry of $B_t$.
The polynomials $F^{B; t_0}_{j,t}(\yy)$ are called the \emph{$F$-polynomials} at $t$.
\end{Definition}
The fact that seemingly rational functions $F^{B; t_0}_{j; t} (\yy)$ are polynomials follows from the Laurent phenomenon of cluster variables \cite[Proposition~3.6]{fziv}.
We also remark that for any $j\in\{1, \dots, n\}$ and $t\in \TT_n$, $F^{B; t_0}_{j; t} (\yy)$ is an element of $\QQsf(y_1, \dots, y_n)$ because of~\eqref{eq:f-initialcondition} and~\eqref{eq:f-frontmutation}.

In this paper, we refer to the recurrence relations (``mutations'') \eqref{eq:c-frontmutation}, \eqref{eq:g-frontmutation} and \eqref{eq:f-frontmutation} as the \emph{final-seed mutations} in contrast to the initial-seed mutations appearing later. Abusing of notation, we denote $C_{t'}^{B;{t_0}}=\mu_\ell\big(C_{t}^{B;{t_0}}\big)$, $G_{t'}^{B;{t_0}}=\mu_\ell\big(G_{t}^{B;{t_0}}\big)$, and $F_{j:t'}^{B;{t_0}}(\yy)=\mu_\ell\big(F_{j;t}^{B;{t_0}}(\yy)\big) $.

The $C$-matrices, the $G$-matrices, and the $F$-polynomials are important because they factorize cluster variables and coefficients by the following formulas:

\begin{Proposition}[separation formulas {\cite[Proposition 3.13, Corollary 6.3]{fziv}}]\label{pr:separation}
Let $\{\Sigma_t\}_{t\in\TT_n}$ be a cluster pattern with coefficient in~$\PP$ with the initial seed~\eqref{initialseed}. Then, for any $t \in \TT_n$ and $j \in \{1, \dots, n\}$, we have
\begin{gather}\label{eq:xjt=F/F}
x_{j;t} = \left( \prod_{k=1}^n x_k^{g^{B; t_0}_{kj;t}} \right) \frac{F_{j;t}^{B;t_0}|_\Fcal(\hat y_1, \dots, \hat y_n)}{F_{j;t}^{B;t_0}|_\PP (y_1, \dots, y_n)}, \\
\label{eq:Y-F}
y_{j;t}=\prod_{k=1}^n y_{k}^{c^{B; t_0}_{kj;t}}\prod_{k=1}^n \big(F_{k;t}^{B;t_0}|_{\mathbb{P}}(y_{1}, \dots, y_{n})\big)^{b_{kj;t}},
\end{gather}
where
\begin{gather*}
\hat y_i = y_i \mathop{\prod}\limits_{j=1}^n x_j^{b_{ji}},
\end{gather*}
and $g_{ij;t}^{B;t_0}$ and $c_{ij;t}^{B;t_0}$ are the $(i,j)$ entry of $G_t^{B;t_0}$ and $C_t^{B;t_0}$, respectively. Also, the rational function $F_{j;t}^{B;t_0}|_\Fcal(\hat y_1, \dots, \hat y_n)$ is the element of $\Fcal$ obtained by substituting $\hat{y}_i$ for $y_i$ in $F_{j;t}^{B;t_0}(y_1, \dots, y_n)$.
\end{Proposition}

The following fact is well-known:
\begin{Proposition} \label{pr:F=1,coefc}\quad
\begin{enumerate}\itemsep=0pt
\item[$(1)$] {\rm \cite[equation (5.5)]{fziv}} For any $j \in\{1, \dots, n\}$ and $t\in \TT_n$, we have
\begin{gather}
\label{eq:F=1} F_{j;t}^{B;t_0}|_{\text{\rm Trop}(y_1, \dots, y_n)}(y_1, \dots, y_n)=1.
\end{gather}
\item[$(2)$] {\rm \cite[equation (2.13)]{fziv}} Let $\{\Sigma_t\}_{t\in\TT_n}$ be a cluster pattern which has principal coefficients at $t_0$.
Then, for any $j \in\{1, \dots, n\}$ and $t\in \TT_n$, we have
\begin{gather}
\label{eq:c-coef} y_{j;t}=\prod_{k=1}^n y_k^{c^{B; t_0}_{kj;t}}.
\end{gather}
\end{enumerate}
\end{Proposition}

\begin{proof}Because of \eqref{eq:f-initialcondition} and \eqref{eq:f-frontmutation}, by using the tropicalization
\begin{align*}
	\pi\colon \ \QQsf(y_1, \dots, y_n) &\longrightarrow \text{Trop}(y_1, \dots, y_n),\\
	y_i &\longmapsto y_i,
\end{align*}
we have \eqref{eq:F=1}.
Moreover, setting $\PP=\text{Trop}(y_1, \dots, y_n)$ in \eqref{eq:Y-F}, we obtain \eqref{eq:c-coef} by \eqref{eq:F=1}.
\end{proof}

Next, we introduce another family of matrices, which are the ``degree matrices'' of the $F$-polynomials.

\begin{Definition}Let $B$ be any initial exchange matrix at $t_0$. For $i\in\{1, \dots, n\}$ and $t\in\TT_n$, let $f^{B; t_0}_{1i;t},\dots, f^{B; t_0}_{ni;t}$ be the maximal degrees of $y_1,\dots,y_n$ in the $i$th $F$-polynomial $F^{B; t_0}_{i;t}(y_1,\dots,y_n)$, respectively. We call the nonnegative integer vectors $\mathbf{f}^{B;t_0}_{i;t}=\left[\begin{smallmatrix}
f_{1i;t}^{B;t_0} \\
\vdots\\
f_{ni;t}^{B;t_0}
\end{smallmatrix}\right]$ the \emph{$\ff$-vectors} at $t$.
We also call the nonnegative integer $n\times n$ matrix $F_{t}^{B;t_0}$ with columns $\ff_{1;t}^{B;t_0}, \dots, \ff_{n;t}^{B;t_0}$ the \emph{$F$-matrix} at $t$.
\end{Definition}
To avoid confusion the notation of $F$-matrices and $F$-polynomials, when we write $F$-po\-ly\-no\-mials, we always write it with arguments.

We have the following description of the $F$-matrices: Consider a semifield homomorphism
\begin{align*}
	\pi\colon \ \QQsf(y_1, \dots, y_n) &\longrightarrow \text{Trop}\big(y_1^{-1}, \dots, y_n^{-1}\big),\\
	y_i &\longmapsto y_i.
\end{align*}
Then, we have
\begin{gather}
\pi\big(F^{B; t_0}_{\ell; t}(y_1,\dots,y_n)\big)=F^{B; t_0}_{\ell; t}|_{\text{Trop}\left(y_1^{-1}, \dots, y_n^{-1}\right)}(y_1, \dots, y_n) \nonumber\\
\hphantom{\pi\big(F^{B; t_0}_{\ell; t}(y_1,\dots,y_n)\big)}{} =\prod_{i=1}^n \big(y_i^{-1}\big)^{-f_{i\ell;t}^{B;t_0}}=\prod_{i=1}^n y_i^{f_{i\ell;t}^{B;t_0}}.\label{eq:tropicalmaximaldegree}
\end{gather}

The $F$-matrices are uniquely determined by the following recurrence relations:

\begin{Proposition} \label{pr:ffnc}
Let $B$ be any initial exchange matrix at $t_0$.
Then, the $F$-matrices have the following recurrence: The initial condition is
\begin{align} \label{eq:fmat-initialcondition}
F^{B; t_0}_{t_0}={O_n},
\end{align}
{where $O_n$ is the zero matrix}. For any edge \begin{xy}(0,1)*+{t}="A",(10,1)*+{t'}="B",\ar@{-}^{\ell}"A";"B" \end{xy} in $\TT_n$, we have the recurrence relation,
\begin{align}\label{eq:fmat-frontmutation}
	F^{B;t_0}_{t'}=F^{B;t_0}_t J_{\ell}+\max\big(\big[C^{B; t_0}_t\big]^{\bullet \ell}_+ +F^{B; t_0}_t[B_t]^{\bullet \ell}_+, \big[{-}C^{B; t_0}_t\big]^{\bullet \ell}_+ +F^{B; t_0}_t[-B_t]^{\bullet \ell}_+\big).
\end{align}
\end{Proposition}
\begin{proof}
Clearly, \eqref{eq:fmat-initialcondition} is obtained by \eqref{eq:f-initialcondition}.
Moreover, applying \eqref{eq:tropicalmaximaldegree} to \eqref{eq:f-frontmutation}, we have \eqref{eq:fmat-frontmutation}.
\end{proof}

We call the recurrence relation~\eqref{eq:fmat-frontmutation} the \emph{final-seed mutations} for the $F$-matrices. As with the $C$-, $G$-matrices and the $F$-polynomials, we denote $F^{B; t_0}_{t'}=\mu_\ell\big(F^{B; t_0}_t\big)$.

Under the exchange of the initial exchange matrices $B$ and $-B$ at $t_0$,
we have the simple relations between the $C$-, $G$- and $F$-matrices.

\begin{Theorem} \label{thm:C=C+FB}
We have the following relations:
\begin{gather} \label{eq:c-}
C^{-B; t_0}_t=C^{B; t_0}_t+F^{B; t_0}_tB_t,\\
\label{eq:G=G+BF}
G_t^{-B;t_0}=G_t^{B;t_0}+BF_t^{B;t_0},\\
\label{eq:F=F}
F_t^{-B;t_0}=F_t^{B;t_0}.
\end{gather}
\end{Theorem}

\begin{proof}Let $\big\{\Sigma_t^{B}=(\xx_t, \yy_t, B_t)\big\}_{t\in\TT_n}$ be a cluster pattern with coefficients in any semifield $\PP$ with the initial seed $(\xx, \yy, B)$. Also, let $\big\{\Sigma_t^{-B}=(\xx'_t, \yy'_t, B'_t)\big\}_{t\in\TT_n}$ be a cluster pattern with coefficients in~$\PP$ with the initial seed $\big(\xx, \yy^{-1}, -B\big)$. Then, by the definition of the mutation~\eqref{eq:matrix-mutation},~\eqref{eq:y-mutation} and~\eqref{eq:x-mutation}, we have \cite[Proof of Proposition~5.3]{fziv}
\begin{gather} \label{x=x,y=y^{-1},B=-B}
\xx'_t=\xx_t, \qquad \yy'_t=\yy_t^{-1}, \qquad B'_t=-B_t.
\end{gather}
We also note that for the initial seed $ \Sigma^{-B}_{t_0}=(\xx', \yy', B')=\big(\xx, \yy^{-1}, -B\big)$, we have
\begin{gather*}
\hat{y'_i}:=y_i'\prod_{j=1}^{n} {x'_j}^{b_{ji}'}=y_i^{-1}\prod_{j=1}^n {x_j}^{-b_{ji}}=\hat{y}_i^{-1}.
\end{gather*}
Now we set $\PP=\trop\big(y_1^{-1}, \dots, y_n^{-1}\big)$ and apply the separation formulas \eqref{eq:xjt=F/F} and \eqref{eq:Y-F} to~\eqref{x=x,y=y^{-1},B=-B}. Then, we obtain
\begin{gather}
\left(\prod_{k=1}^n x_k^{g^{-B; t_0}_{kj;t}}\right) F_{j;t}^{-B;t_0}|_\Fcal\big(\hat{y}^{-1}_1, \dots, \hat{y}^{-1}_n\big)\nonumber\\
\qquad{} =\left(\prod_{k=1}^n x_k^{g^{B; t_0}_{kj;t}}\right) \frac{F_{j;t}^{B;t_0}|_\Fcal(\hat y_1, \dots, \hat y_n)}{F_{j;t}^{B;t_0}|_{\text{Trop}\left(y^{-1}_1, \dots, y^{-1}_n\right)} (y_1, \dots, y_n)}, \label{eq:degree-rescaling}\\
\label{eq:Y-F-1}
\left(\prod_{k=1}^n \big(y^{-1}_k\big)^{c^{-B; t_0}_{kj;t}}\right)^{-1}
=\prod_{k=1}^n y_{k}^{c^{B; t_0}_{kj;t}}\prod_{k=1}^n \big(F_{k;t}^{B;t_0}|_{\text{Trop}\left(y_1^{-1}, \dots, y_n^{-1}\right)}(y_{1}, \dots, y_{n})\big)^{b_{kj;t}},
\end{gather}
by \eqref{eq:F=1} and \eqref{eq:c-coef}.
Applying \eqref{eq:tropicalmaximaldegree} to \eqref{eq:Y-F-1}, we obtain
\begin{gather*}
\cc_{j;t}^{-B;t_0} = \cc_{j;t}^{B;t_0}+\sum_{i=1}^n b_{ij; t}\ff^{B;t_0}_{i;t} ,
\end{gather*}
thus we have \eqref{eq:c-}.
To show \eqref{eq:G=G+BF} from \eqref{eq:degree-rescaling}, let us set the $\ZZ^n$-gradings in $\ZZ[x_1^{\pm1}, \dots, x_n^{\pm1}, \allowbreak y_1, \dots, y_n]$ as follows \cite[equations~(6.1) and~(6.2)]{fziv}:
\begin{gather*}
\deg(x_i)=\mathbf{e}_i, \qquad \deg(y_i)=-{\mathbf b}_i,
\end{gather*}
where $\mathbf{e}_i$ is the $i$th column vector of $I_n$ and ${\mathbf b}_i$ is the $i$th column vector of $B$.
Then, we have
\begin{gather*}
\deg(\hat{y}_i)=0.
\end{gather*}
Hence comparing the $\ZZ^n$-gradings of both sides of~\eqref{eq:degree-rescaling}, we obtain
\begin{gather*} 
\gg_{j;t}^{-B;t_0} = \gg_{j;t}^{B;t_0}+\sum_{i=1}^n f^{B; t_0}_{ij; t}\mathbf{b}_i .
\end{gather*}
Therefore, we have \eqref{eq:G=G+BF}. Let us prove \eqref{eq:F=F}. Substituting $x_i=1$ $(i=1, \dots, n)$ for~\eqref{eq:degree-rescaling}, we have \cite[equation~(5.6)]{fziv}
\begin{gather}\label{Fpoly=Fpoly}
F_{j;t}^{-B;t_0}\big(y^{-1}_1,\dots,y^{-1}_n\big)
= \frac{F_{j;t}^{B;t_0}(y_1,\dots,y_n)}
{F_{j;t}^{B;t_0}|_{\text{Trop}\left(y^{-1}_1, \dots, y^{-1}_n\right)}(y_1, \dots,y_n)}.
\end{gather}
Under the exchange of $B$ and $-B$ in \eqref{Fpoly=Fpoly}, we also have
\begin{gather}\label{Fpoly=Fpoly2}
F_{j;t}^{B;t_0}(y_1,\dots,y_n)
= \frac{F_{j;t}^{-B;t_0}\big(y^{-1}_1,\dots,y^{-1}_n\big)}
{F_{j;t}^{-B;t_0}|_{\text{Trop}(y_1, \dots, y_n)}\big(y^{-1}_1, \dots,y^{-1}_n\big)}.
\end{gather}
Hence combining \eqref{Fpoly=Fpoly} with \eqref{Fpoly=Fpoly2}, we obtain
\begin{gather}\label{eq:FF=1}
F_{j;t}^{-B;t_0}|_{\text{Trop}(y_1, \dots, y_n)}\big(y_1^{-1},\dots,y_n^{-1}\big)F_{j;t}^{B;t_0}|_{\text{Trop}\left(y_1^{-1}, \dots, y_n^{-1}\right)}(y_1,\dots,y_n)=1.
\end{gather}
Comparing exponents of $y_i$ of both sides of \eqref{eq:FF=1}, we have
\begin{gather*}
-\ff_{j;t}^{-B;t_0}+\ff_{j;t}^{B;t_0}=\mathbf0,
\end{gather*}
thus we obtain \eqref{eq:F=F}.
\end{proof}

Thanks to the relation \eqref{eq:c-}, we have the following alternating expression of the final-seed mutations of the $F$-matrices:

\begin{Proposition}\label{pr:ffr}Let $\varepsilon\in\{\pm 1\}$.
For any edge \begin{xy}(0,1)*+{t}="A",(10,1)*+{t'}="B",\ar@{-}^{\ell}"A";"B" \end{xy} in $\TT_n$, the matrices $F^{B; t_0}_t$ and $F^{B; t_0}_{t'}$ are related by
\begin{gather}\label{eq:ffr}
	F^{B; t_0}_{t'}=F^{B; t_0}_t\big(J_{\ell}+[-\varepsilon B_t]^{\bullet \ell}_+\big)+\big[{-}\varepsilon C^{B; t_0}_t\big]^{\bullet \ell}_+ +\big[\varepsilon C^{-B; t_0}_t\big]^{\bullet\ell}_+.
\end{gather}
\end{Proposition}

\begin{proof}Firstly, we prove the case of $\varepsilon=1$ in \eqref{eq:ffr}. By \eqref{eq:b--b}, \eqref{eq:fmat-frontmutation} and \eqref{eq:c-}, we have
\begin{align*}
	F^{B;t_0}_{t'}&=F^{B;t_0}_t\big(J_{\ell}+[-B_t]^{\bullet \ell}_+\big)+\max\big(\big[C^{B; t_0}_t\big]^{\bullet \ell}_+ +F^{B; t_0}_tB_t^{\bullet \ell}, \big[{-}C^{B; t_0}_t\big]^{\bullet \ell}_+\big)\\
	&=F^{B;t_0}_t\big(J_{\ell}+[-B_t]^{\bullet \ell}_+\big)+\max\big(\big[C^{B; t_0}_t\big]^{\bullet \ell}_+ +\big(C^{-B; t_0}_t\big)^{\bullet \ell}-\big(C^{B; t_0}_t\big)^{\bullet\ell} ,\big[{-}C^{B; t_0}_t\big]^{\bullet \ell}_+\big)\\
	&=F^{B;t_0}_t\big(J_{\ell}+[-B_t]^{\bullet \ell}_+\big)+\max\big(\big[{-}C^{B; t_0}_t\big]^{\bullet \ell}_++\big(C^{-B; t_0}_t\big)^{\bullet \ell},\big[{-}C^{B; t_0}_t\big]^{\bullet \ell}_+\big)\\
	&=F^{B; t_0}_t\big(J_{\ell}+[-B_t]^{\bullet \ell}_+\big)+\big[{-}C^{B; t_0}_t\big]^{\bullet \ell}_+ +\big[ C^{-B; t_0}_t\big]^{\bullet\ell}_+
\end{align*}
as desired.
Secondly, we prove the case of $\varepsilon=-1$ in \eqref{eq:ffr}. In the same way as $\varepsilon=1$, we have
\begin{align*}
	F^{B;t_0}_{t'}&=F^{B;t_0}_t\big(J_{\ell}+[B_t]^{\bullet \ell}_+\big)+\max\big(\big[C^{B; t_0}_t\big]^{\bullet \ell}_+, \big[{-}C^{B; t_0}_t\big]^{\bullet \ell}_+ -F^{B; t_0}_tB_t^{\bullet \ell}\big)\\
	&=F^{B;t_0}_t\big(J_{\ell}+[B_t]^{\bullet \ell}_+\big)+\max\big(\big[C^{B; t_0}_t\big]^{\bullet \ell}_+ ,\big[C^{B; t_0}_t\big]^{\bullet \ell}_+-\big(C^{-B; t_0}_t\big)^{\bullet \ell}\big)\\
	&=F^{B; t_0}_t\big(J_{\ell}+[B_t]^{\bullet \ell}_+\big)+\big[C^{B; t_0}_t\big]^{\bullet \ell}_+ +\big[{-} C^{-B; t_0}_t\big]^{\bullet\ell}_+
\end{align*}
as desired.
\end{proof}

\subsection[Final-seed mutations with sign-coherence of $C$-matrices]{Final-seed mutations with sign-coherence of $\boldsymbol{C}$-matrices}\label{section2.3}

In this subsection, we reduce the final-seed mutation formulas by applying the sign-coherence of the $C$-matrices.

\begin{Definition}Let $A$ be an (not necessarily square) integer matrix. We say that $A$ is \emph{column sign-coherent} (resp.\ \emph{row sign-coherent}) if for any column (resp.\ row) of $A$, its entries are either all nonnegative, or all nonpositive, and not all zero.
\end{Definition}

When $A$ is column sign-coherent (resp.\ row sign-coherent), we can define its $\ell$th \emph{column sign}~$\varepsilon_{\bullet\ell}(A)$ (resp.\ \emph{row sign} $\varepsilon_{\ell\bullet}(A)$) as the sign of nonzero entries of the $\ell$th column (resp. row) of $A$. We have the following fundamental and nontrivial result:

\begin{Theorem}[{\cite[Corollary 5.5]{ghkk}}]\label{thm:signs-ci} For any initial exchange matrix $B$, every $C$-matrix $C_t^{B;t_0}$ $(t\in \TT_n)$ is column sign-coherent.
\end{Theorem}

The column signs of the $C$-matrix $C_t^{B;t_0}$ are called the \emph{tropical signs} due to \eqref{eq:c-coef}. Using them, the following reduced expression of the final-seed mutations of the $C$- and $G$-matrices are obtained:

\begin{Proposition} [{\cite[Proposition 1.3]{nz}}] \label{pr:cgfs}
For any edge \begin{xy}(0,1)*+{t}="A",(10,1)*+{t'}="B",\ar@{-}^{\ell}"A";"B" \end{xy} in $\TT_n$, we have
\begin{gather}\label{eq:c-frontmutationsign}
C^{B; t_0}_{t'}=C^{B; t_0}_t\big(J_{\ell}+\big[\varepsilon_{\bullet\ell}\big(C^{B; t_0}_t\big) B_t\big]^{\ell\bullet}_+\big), \\
\label{eq:g-frontmutationsign}
G^{B; t_0}_{t'}=G^{B; t_0}_t\big(J_{\ell}+\big[{-}\varepsilon_{\bullet\ell}(C^{B; t_0}_t) B_t\big]^{\bullet \ell}_+\big).
\end{gather}
\end{Proposition}
They are obtained from \eqref{eq:c-frontmutation} and \eqref{eq:g-frontmutation} by setting $\varepsilon=\varepsilon_{\bullet\ell}\big(C_{t}^{B;t_0}\big)$ and $\varepsilon=-\varepsilon_{\bullet\ell}\big(C_{t}^{B;t_0}\big)$, respectively.

The following fact is shown by \cite{fziv}:

\begin{Proposition}[{\cite[Proposition 5.6]{fziv}}] \label{eqfc} For any initial exchange matrix $B$, the following are equivalent:
\begin{enumerate}\itemsep=0pt
\item[$(i)$] The sign-coherence of the $C$-matrices holds.
\item[$(ii)$] Every polynomial $F_{\ell;t}^{B;t_0}(\yy)$ has constant term~$1$.
\item[$(iii)$] Every polynomial $F_{\ell;t}^{B;t_0}(\yy)$ has a unique monomial of maximal degree. Furthermore, this monomial has coefficient~$1$, and it is divisible by all the other occurring monomials.
\end{enumerate}
\end{Proposition}

{In proposition \ref{eqfc}, the equivalence of (ii) and (iii) is proved by \cite[Proposition 5.3]{fziv} (see \cite[Conjectures 5.4 and 5.5]{fziv}).}

\begin{Remark}
In the definition of the (column) sign-coherence in \cite{fziv}, the nonzero vector property of column vectors are not assumed. However, this property can be easily recovered, since $\det C_t^{B;t_0}=\pm1$ due to \eqref{eq:c-frontmutationsign}.
\end{Remark}

By Theorem \ref{thm:signs-ci} and Proposition \ref{eqfc}, we have the following description of the $\ff$-vectors:

\begin{Corollary}
The $\ff$-vector $\ff_{i;t}^{B;t_0}$ is the exponent vector of the unique monomial with maximal degree of $F_{i;t}^{B;t_0}(\yy)$. In other words, the unique monomial with maximal degree of $F_{i;t}^{B;t_0}(\yy)$ is given by $y_1^{f_{1i;t}^{B;t_0}}\cdots y_n^{f_{ni;t}^{B;t_0}}$.
\end{Corollary}

Now, let us give the reduced expression of the final-seed mutations of the $F$-matrices by using the tropical signs.

\begin{Proposition} \label{pr:ffront}For any edge \begin{xy}(0,1)*+{t}="A",(10,1)*+{t'}="B",\ar@{-}^{\ell}"A";"B" \end{xy} in $\TT_n$, we have
{\begin{align}
	F^{B; t_0}_{t'}&=F^{B; t_0}_t \big(J_{\ell}+\big[\varepsilon_{\bullet\ell}\big(C^{-B; t_0}_t\big) B_t\big]^{\bullet \ell}_+\big)+\big[\varepsilon_{\bullet\ell}\big(C^{-B; t_0}_t\big)C^{B; t_0}_t\big]^{\bullet {\ell}}_+ \notag\\
	&=F^{B; t_0}_t \big(J_{\ell}+\big[{-}\varepsilon_{\bullet\ell}\big(C^{B; t_0}_t\big)B_t\big]^{\bullet \ell}_+\big)+\big[\varepsilon_{\bullet\ell}\big(C^{B; t_0}_t\big)C^{-B; t_0}_t\big]^{\bullet {\ell}}_+. \label{eq:ffront}
\end{align}}
\end{Proposition}

\begin{proof}Substituting $\varepsilon=-\varepsilon_{\bullet \ell} \big(C^{-B; t_0}_t\big)$ or $\varepsilon=\varepsilon_{\bullet \ell}\big(C^{B; t_0}_t\big)$ for~\eqref{eq:ffr}, we obtain~\eqref{eq:ffront}.
\end{proof}

\section{Initial-seed mutations}\label{section3}
\subsection[Initial-seed mutations of functions $\Ycal$ and $\Xcal$]{Initial-seed mutations of functions $\boldsymbol{\Ycal}$ and $\boldsymbol{\Xcal}$}\label{section3.1}

We introduce the concept of the \emph{initial-seed mutations} which appears in \cite{fziv, nz, rs} in the following way. Let $\QQ_{\text{sf}}(\yy)$ be the universal semifield with formal variables $\yy=(y_1,\dots,y_n)$ in Section~\ref{section2.1}. Let $\{\Sigma_t\}_{t\in\TT_n}$ be the cluster pattern with coefficients in $\QQ_{\text{sf}}(\yy)$ where the initial coefficients $\yy_{t_0}$ are taken as the above formal variables $\yy$. Then, recursively applying the mutations~\eqref{eq:y-mutation} or~\eqref{eq:y-mutation2} from the initial coefficients, $y_{i;t}$ are written as a rational function of~$\yy$:
\begin{gather*}
y_{i;t}=\Ycal_{i;t}^{B;t_0}(\yy)\in \QQ_{\text{sf}}(\yy).
\end{gather*}
Similarly, recursively applying the mutations \eqref{eq:x-mutation}, $x_{i;t}\in\Fcal$ are written as a rational function of the initial cluster variables $\xx_{t_0}=\xx$ with coefficients in $\QQ(\QQ_{\text{sf}}(\yy))$:
\begin{gather*}
x_{i;t}=\Xcal_{i;t}^{B;t_0}(\xx)\in \QQ(\QQ_{\text{sf}}(\yy))(\xx).
\end{gather*}
Then, for any cluster pattern $\{\Sigma_t\}_{t\in\TT_n}$ with coefficients in $\PP$, we recover $x_{i;t}$ and $y_{i;t}$ by the specialisation $\pi\colon \QQ_{\text{sf}}(\yy)\rightarrow\PP$ with $y_i$ setting to be the initial coefficients of $\Sigma_{t_0}$. Let $t_1\in\TT_n$ be the vertex with \begin{xy}(0,1)*+{t_0}="A",(10,1)*+{t_1}="B",\ar@{-}^{k}"A";"B" \end{xy}and let $B_1=\mu_k(B)$. Then, the rational functions $\Ycal_{i;t}^{B;t_0}(\yy)$ and~$\Ycal_{i;t}^{B_1;t_1}(\yy)$ are related by
\begin{gather*}
\Ycal_{i;t}^{B_1;t_1}(\yy)=\rho_k\big(\Ycal_{i;t}^{B;t_0}(\yy)\big),
\end{gather*}
where $\rho_k$ is the semifield automorphism of $\QQ_{\text{sf}}(\yy)$ defined by
\begin{align*}
\rho_k\colon \ \QQ_{\text{sf}}(\yy)&\rightarrow \QQ_{\text{sf}}(\yy), \notag \\
y_j&\mapsto
\begin{cases}
y_{k}^{-1} &\text{if $j=k$,} \\
y_j y_k^{[b_{kj}]_+}(y_k \oplus 1)^{-b_{kj}} &\text{otherwise.}
\end{cases}
\end{align*}
Similarly, the rational functions $\Xcal_{i;t}^{B;t_0}(\xx)$ and $\Xcal_{i;t}^{B_1;t_1}(\xx)$ are related by
\begin{gather*}
\Xcal_{i;t}^{B_1;t_1}(\xx)=\rho_k\big(\Xcal_{i;t}^{B;t_0}(\xx)\big),
\end{gather*}
where $\rho_k$ is the field automorphism of $\QQ(\QQ_{\text{sf}}(\yy))(\xx)$ defined by
\begin{align*}
\rho_k\colon \ \QQ(\QQ_{\text{sf}}(\yy))(\xx)&\rightarrow \QQ(\QQ_{\text{sf}}(\yy))(\xx), \notag \\
y_j&\mapsto
\begin{cases}
y_{k}^{-1} &\text{if $j=k$}, \\
y_j y_k^{[b_{kj}]_+}(y_k \oplus 1)^{-b_{kj}} &\text{otherwise},
\end{cases}\\
x_j&\mapsto \begin{cases}\dfrac{y_k\mathop{\prod}\limits_{i=1}^n x_i^{[b_{ik}]_+}+\mathop{\prod}\limits_{i=1}^n x_i^{[-b_{ik}]_+}}{(y_k\oplus 1)x_k} &\text{if $j=k$},\\
x_j &\text{otherwise}.
\end{cases}
\end{align*}
We call them the \emph{initial-seed mutations of the functions $\Ycal$ and $\Xcal$}.

\subsection[Initial-seed mutations without sign-coherence of $C$-matrices]{Initial-seed mutations without sign-coherence of $\boldsymbol{C}$-matrices}\label{section3.2}
We use the notations in Section 3.1 continuously. By Proposition \ref{pr:separation}, we have
\begin{gather*}
\Xcal_{j;t}^{B;t_0}(\xx) = \left(\prod_{k=1}^n x_k^{g^{B;t_0}_{kj;t}}\right) \frac{F_{j;t}^{B;t_0}(\hat y_1, \dots, \hat y_n)}{F_{j;t}^{B;t_0}(y_1, \dots, y_n)}, \\
\Ycal_{j;t}^{B;t_0}(\yy) =\prod_{k=1}^n y_{k}^{c^{B;t_0}_{kj;t}}
\prod_{k=1}^n \big(F_{k;t}^{B;t_0}(y_{1}, \dots, y_{n})\big)^{b_{kj;t}}.
\end{gather*}
As with the final-seed mutation, we will define \emph{the initial-seed mutations in direction $k$ of the $C$-matrices} (resp.\ \emph{the $G$-matrices, the $F$-polynomials, the $F$-matrices}) as transformations from~$C_t^{B;t_0}$ to $C_t^{B_1;t_1}$ (resp.\ from $G_t^{B;t_0}$ to $G_t^{B_1;t_1}$, from $F_{j;t}^{B;t_0}$ to $F_{j;t}^{B_1;t_1}$, from~$F_t^{B;t_0}$ to~$F_t^{B_1;t_1}$).
We will deduce the initial-seed mutations of the $C$-, $G$-matrices, the $F$-polynomials and the $F$-matrices. In order to describe these initial-seed mutations, we introduce the \emph{$H$-matrices} according to~\cite{fziv}.

\begin{Definition} [{\cite[equation~(6.16)]{fziv}}] Let $B$ be any initial exchange matrix at $t_0$. Then, for any $t$, the $(i, j)$ entry of $H^{B; t_0}_t=\big(h_{ij; t}^{B;t_0}\big)$ is given by
\begin{gather*}
	u^{h_{ij; t}^{B;t_0}}=F^{B; t_0}_{j;t}|_{\text{Trop}(u)}\big(u^{[-b_{i1}]_+}, \dots, u^{-1}, \dots, u^{[-b_{in}]_+}\big) \qquad \big(\text{$u^{-1}$ in the $i$th position}\big).
\end{gather*}
The matrix $H^{B; t_0}_t$ is called the \emph{$H$-matrix} at $t$. \end{Definition}The following fact holds \cite[Proof of Proposition~6.8]{fziv}.
\begin{Lemma} \label{lem:horigin} We have the following equality:
\begin{align}\label{eq:horigin}
 {y'_k}^{h_{kj; t}^{B;t_0}}=F^{B; t_0}_{j;t}|_{{\rm Trop}(y'_1, \dots, y'_n)}(y_1, \dots, y_n),
\end{align}
where $(y'_1, \dots, y'_n)$ are the coefficients at $t_1$ connected with $t_0$ by an edge labeled $k$ in $\TT_n.$
\end{Lemma}

\begin{proof}Consider the cluster pattern with coefficients in $\trop(y'_1, \dots, y'_n)$. Let $\yy=(y_1, \dots, y_n)$ be the coefficients at $t_0$. Then, $\yy$ and $\yy'$ have the following relation:
\begin{align} \label{eq:y=y'}
	y_i = \begin{cases}
		{y'_k}^{-1} & \text{if} \ i=k, \\
		{y'_i}{y'_k}^{[-b_{ki}]_+} & \text{if} \ i\neq k.
	\end{cases}
\end{align}
Therefore, for any $j\in\{1, \dots, n\}$, we have
\begin{align*}
F^{B; t_0}_{j;t}|_{\text{Trop}(y'_1, \dots, y'_n)}(y_1, \dots, y_n)&\;\:=F^{B; t_0}_{j;t}|_{\text{Trop}(y'_1, \dots, y'_n)}\big(y'_1{y'_k}^{[-b_{k1}]_+}, \dots, {y'_k}^{-1}, \dots, y'_n{y'_k}^{[-b_{kn}]_+}\big)\\
&\overset{\eqref{eq:F=1}}{=}F^{B; t_0}_{j;t}|_{\text{Trop}(y'_k)}\big({y'_k}^{[-b_{k1}]_+}, \dots, {y'_k}^{-1}, \dots, {y'_k}^{[-b_{kn}]_+}\big)\\
&\;\:={y'_k}^{h_{kj; t}^{B;t_0}}.\tag*{\qed}
\end{align*}\renewcommand{\qed}{}
\end{proof}

The initial-seed mutations of the $C$- and $G$-matrices are given as follows, where the latter was given in \cite[Proposition 6.8]{fziv}:

\begin{Proposition} \label{pr:crear}
Let \begin{xy}(0,1)*+{t_0}="A",(10,1)*+{t_1}="B",\ar@{-}^{k}"A";"B" \end{xy} in $\TT_n$, $\mu_k (B)=B_1$ and $\varepsilon \in\{\pm 1\}$. Then, for any $t$, we have
\begin{gather}\label{eq:crear}
C^{B_1; t_1}_t=\big(J_k+[-\varepsilon B]^{k\bullet}_+\big)C^{B; t_0}_t+H_t(\varepsilon)^{k\bullet} B_t, \\
\label{eq:grear}
G^{B_1; t_1}_t=\big(J_k+[\varepsilon B]^{\bullet k}_+\big)G^{B; t_0}_t-BH_t(\varepsilon)^{k\bullet},
\end{gather}
where $H_t(+)=H^{B; t_0}_t$, $H_t(-)=H^{B_1; t_1}_t$.
\end{Proposition}

\begin{proof}We denote $C^{B; t_0}_t=(c_{ij;t})$, $ C^{B_1; t_1}_t=(c'_{ij;t})$, $H^{B; t_0}_t=(h_{ij; t})$, and $H^{B_1; t_1}_t=(h'_{ij; t})$. The equation~\eqref{eq:grear} is just \cite[Proposition~6.8]{fziv}, rewritten in matrix form.

Let us show \eqref{eq:crear}. Consider the same cluster pattern as in the proof of Lemma~\ref{lem:horigin}. Then, applying~\eqref{eq:Y-F} and~\eqref{eq:c-coef} to any coefficient $y_{j;t}$, we have
\begin{gather} \label{eq:y'=y}
\prod_{i=1}^n {y'_i}^{c'_{ij; t}}=\prod_{i=1}^n {y_i}^{c_{ij; t}} \prod_{i=1}^n F^{B; t_0}_{i; t}|_{\text{Trop}(y'_1, \dots, y'_n)}(y_1, \dots, y_n)^{b_{ij; t}}.
\end{gather}
Substituting \eqref{eq:y=y'} for \eqref{eq:y'=y} and using \eqref{eq:horigin}, we have
\begin{gather}
	\prod_{i=1}^n {y'_i}^{c'_{ij; t}}=\left(\prod_{i\neq k} {y'_i}^{c_{ij; t}}{y'_k}^{[-b_{ki}]_+c_{ij; t}}\right) {y'_k}^{-c_{kj; t}} \prod_{i=1}^n {y'_k}^{h_{ki; t}b_{ij; t}}.\label{eq:y=yyyy}
\end{gather}
Comparing exponents of $y'_i$ of the both sides of \eqref{eq:y=yyyy}, we obtain
\begin{gather*}
	c'_{ij; t}=\begin{cases}
\displaystyle -c_{kj; t}+\mathop{\sum}\limits_{\ell=1}^n [-b_{k\ell}]_+c_{\ell j;t} + \mathop{\sum}\limits_{\ell=1}^n h_{k\ell; t}b_{\ell j; t} & \text{if} \ i=k, \\
c_{ij; t} &\text{if} \ i \neq k.
\end{cases}
\end{gather*}
Also, by interchanging $t_0$ and $t_1$, we get
\begin{gather*}
	c'_{ij; t}=\begin{cases}
\displaystyle -c_{kj; t}+\mathop{\sum}\limits_{\ell=1}^n [b_{k\ell}]_+c_{\ell j; t} + \mathop{\sum}\limits_{\ell=1}^n h'_{k\ell; t}b_{\ell j; t} & \text{if} \ i=k, \\
c_{ij; t} &\text{if} \ i\neq k.
\end{cases}
\end{gather*}
Thus we have \eqref{eq:crear}.
\end{proof}

The initial-seed mutations of the $F$-polynomials were given in \cite[Proof of Proposition 6.8]{fziv} as follows:
\begin{Proposition}[{\cite[equation~(6.21)]{fziv}}]\label{pr:fpolyrear}
Let \begin{xy}(0,1)*+{t_0}="A",(10,1)*+{t_1}="B",\ar@{-}^{k}"A";"B" \end{xy} in $\TT_n$ and $\mu_k(B)=B_1$.
Then, for any $j\in\{1, \dots, n\}$ and $t\in\TT_n$, the polynomials $F^{B; t_0}_{j;t}(\yy)$ and $F^{B_1; t_1}_{j;t}(\yy)$ are related by
\begin{align}
F^{B_1; t_1}_{j;t}(y_1, \dots, y_n)&=(1+y_k)^{g_{kj;t}^{B;t_0}}{y_k}^{-h_{kj;t}^{B;t_0}} \label{eq:fpolyrear}\\
& \quad \times F^{B; t_0}_{j;t}\big(y_1{y_k}^{[-b_{k1}]_+}(y_k+1)^{b_{k1}}, \dots, {y_k}^{-1}, \dots, y_n{y_k}^{[-b_{kn}]_+}(y_k+1)^{b_{kn}}\big),\nonumber
\end{align}
where ${y_k}^{-1}$ is in the $k$th position.
\end{Proposition}
The initial-seed mutation of the $F$-matrices are deduced from Proposition \ref{pr:fpolyrear} as follows:
\begin{Proposition}\label{pr:frear}
Let \begin{xy}(0,1)*+{t_0}="A",(10,1)*+{t_1}="B",\ar@{-}^{k}"A";"B" \end{xy} in $\TT_n$, $\mu_k(B)=B_1$ and $\varepsilon \in \{\pm 1\}$. Then for any $t$, the matrices $F^{B; t_0}_t$ and $F^{B_1; t_1}_t$ are related by
\begin{align}\label{eq:frear}
	F^{B_1; t_1}_t&=\big(J_k+[\varepsilon B]^{k\bullet}_+\big)F^{B; t_0}_t + (\varepsilon G^{B; t_0}_t)^{k\bullet} -H^{-B; t_0}_t(\varepsilon)^{k\bullet}-H^{B; t_0}_t(\varepsilon)^{k\bullet}\\
\label{eq:frear2}
&=\big(J_k+[-\varepsilon B]^{k\bullet}_+\big)F^{B; t_0}_t + (\varepsilon G^{-B; t_0}_t)^{k\bullet} -H^{-B; t_0}_t(\varepsilon)^{k\bullet}-H^{B; t_0}_t(\varepsilon)^{k\bullet}.
\end{align}
\end{Proposition}

\begin{proof}Let us show \eqref{eq:frear} in case of $\varepsilon=1$. Substituting \eqref{eq:fpolyrear} for $y_i=y'_i$ and evaluating \eqref{eq:fpolyrear} at $\trop\big({y'_1}^{-1}, \dots, {y'_n}^{-1}\big)$, we have
\begin{align*}
&F^{B_1; t_1}_{j;t}|_{\trop({y'_1}^{-1}, \dots, {y'_n}^{-1})}(y'_1, \dots, y'_n)\\
&\;\:=(1\oplus y'_k)^{g^{B; t_0}_{kj}}{y'_k}^{-h^{B;t_0}_{kj}}\\
&\qquad \times F^{B; t_0}_{j;t}|_{\trop\left({y'_1}^{-1}, \dots, {y'_n}^{-1}\right)}\big(y'_1{y'_k}^{[-b_{k1}]_+}(y'_k\oplus 1)^{b_{k1}}, \dots, {y'_k}^{-1}, \dots, y'_n{y'_k}^{[-b_{kn}]_+}(y'_k\oplus1)^{b_{kn}}\big) \\
&\,\overset{\eqref{eq:y-mutation2}}{=}(1\oplus y'_k)^{g^{B; t_0}_{kj}}{y'_k}^{-h^{B;t_0}_{kj}}\\
&\qquad \times F^{B; t_0}_{j;t}|_{\trop\left({y'_1}^{-1}, \dots, {y'_n}^{-1}\right)}\big(y'_1{y'_k}^{[b_{k1}]_+}\big({y'_k}^{-1}\oplus 1\big)^{b_{k1}}, \dots,\\
& \qquad\qquad {y'_k}^{-1}, \dots, y'_n{y'_k}^{[b_{kn}]_+}({y'_k}^{-1}\oplus1)^{b_{kn}}\big) \\
&\;\:={y'_k}^{g^{B; t_0}_{kj}}{y'_k}^{-h^{B;t_0}_{kj}}F^{B; t_0}_{j;t}|_{\trop\left({y'_1}^{-1}, \dots, {y'_n}^{-1}\right)}\big({y'_1}{y'_k}^{[b_{k1}]_+}, \dots, {y'_k}^{-1}, \dots, y'_n{y'_k}^{[b_{kn}]_+}\big)
\\
&\overset{\eqref{Fpoly=Fpoly2}}{=}{y'_k}^{g^{B; t_0}_{kj}}{y'_k}^{-h^{B;t_0}_{kj}}\frac{F_{j;t}^{-B;t_0}|_{\trop\left({y'_1}^{-1}, \dots, {y'_n}^{-1}\right)}\big({y'_1}^{-1}{y'_k}^{-[b_{k1}]_+}, \dots, y'_k, \dots, {y'_n}^{-1}{y'_k}^{-[b_{kn}]_+}\big)}{F_{j;t}^{-B;t_0}|_{\text{Trop}(y_1, \dots, y_n)}\big(y^{-1}_1, \dots,y^{-1}_n\big)|_{y_i\mapsto {y'_i}{y'_k}^{[b_{ki}]_+}, y_k\mapsto {y'_k}^{-1}}} \\
&\;\;={y'_k}^{g^{B; t_0}_{kj}}{y'_k}^{-h^{B;t_0}_{kj}}{y'_k}^{-h^{-B;t_0}_{kj}}\prod_{i\neq k}\big({y'_i}^{f_{ij}}{y'_k}^{[b_{ki}]_+f_{ij}}\big){y'_k}^{-f_{kj}}.
\end{align*}
Comparing the exponent of both sides, we have
\begin{gather*}
f^{B_1; t_1}_{ij;t}=\begin{cases}
\displaystyle	g^{B; t_0}_{kj;t}-h^{B;t_0}_{kj;t}-h^{-B;t_0}_{kj;t}+\mathop{\sum}\limits_{i=1}^n[b_{ki}]_+ f^{B;t_0}_{ij;t}-f^{B; t_0}_{kj;t} \quad &\text{if} \ i = k,\\
	f^{B; t_0}_{ij;t} \quad &\text{if} \ i \neq k.
\end{cases}
\end{gather*}
Hence we obtain the desired equality \eqref{eq:frear}. Also replacing $B$ with $-B$ in \eqref{eq:frear} and app\-lying~\eqref{eq:F=F} to it, we get \eqref{eq:frear2}.
\end{proof}

\subsection[Initial-seed mutations with sign-coherence of $C$-matrices]{Initial-seed mutations with sign-coherence of $\boldsymbol{C}$-matrices}\label{section3.3}
In this subsection, we reduce the initial-seed mutation formulas by applying the sign-coherence of the $C$-matrices. Let us introduce a duality between the $C$-matrices and the $G$-matrices, which is a result in~\cite{nz},
and give the reduced form of the initial-seed mutations of the $C$- and $G$-matrices.

Under the sign-coherence of the $C$-matrices (Theorem \ref{thm:signs-ci}), we have the following result:
\begin{Proposition} \label{pr:cgrs}\quad
\begin{enumerate}\itemsep=0pt
\item[$(1)$] For any exchange matrix~$B$ and $t_0, t \in \TT_n$, we have
\begin{gather}\label{eq:CG-opponent}
\big(G^{B:t_0}_t\big)^T=C^{B^T_t; t}_{t_0}.
\end{gather}
\item[$(2)$] Let \begin{xy}(0,1)*+{t_0}="A",(10,1)*+{t_1}="B",\ar@{-}^{k}"A";"B" \end{xy} in $\TT_n$, $\mu_k(B)=B_1$ and $\varepsilon\in\{\pm1\}$. Then, we have
\begin{align}
C^{B_1; t_1}_t&=\big(J_k+[-\varepsilon B]^{k\bullet}_+\big)C^{B; t_0}_t-\big[{-}\varepsilon G^{B; t_0}_t\big]^{k\bullet}_+ B_t, \label{eq:crs}\\
G^{B_1; t_1}_t&=\big(J_k+[\varepsilon B]^{\bullet k}_+\big)G^{B; t_0}_t + B\big[{-}\varepsilon G^{B; t_0}_t\big]^{k\bullet}_+. \label{eq:grs}
\end{align}
\item[$(3)$] We have the reduced forms of the initial-seed mutations as follows:
\begin{align}
C^{B_1; t_1}_t&=\big(J_k+\big[{-}\varepsilon_{k\bullet}\big(G^{B; t_0}_t\big) B\big]^{k\bullet}_+\big)C^{B; t_0}_t, \label{eq:crwc}\\
G^{B_1; t_1}_t&=\big(J_k+\big[\varepsilon_{k\bullet}\big(G^{B; t_0}_t\big) B\big]^{\bullet k}_+\big)G^{B; t_0}_t. \label{eq:grwc}
\end{align}
\end{enumerate}
\end{Proposition}

\begin{proof} Equalities \eqref{eq:CG-opponent} and \eqref{eq:grs} are the results in \cite[equations~(1.13) and (4.1)]{nz}. Also, \eqref{eq:crwc} is obtained by combining~\eqref{eq:CG-opponent} with~\cite[equations~(1.16) and (2.7)]{nz}. Using~\eqref{eq:CG-opponent}, we have~\eqref{eq:crs} by~\eqref{eq:g-frontmutation}. We note that the $G$-matrices have the row sign-coherence by~\eqref{eq:CG-opponent}. By substituting $\varepsilon =\varepsilon_{k\bullet}\big(G^{B; t_0}_t\big)$ for~\eqref{eq:grs}, we obtain~\eqref{eq:grwc}.
\end{proof}

Through the duality \eqref{eq:CG-opponent}, we can find out the dual equalities between the unreduced form of the final-seed and initial-seed mutations, \eqref{eq:g-frontmutation} and \eqref{eq:crs}, \eqref{eq:c-frontmutation} and \eqref{eq:grs}, respectively. Similarly, the reduced form of the final-seed and initial-seed mutations \eqref{eq:c-frontmutationsign} and \eqref{eq:grwc}, \eqref{eq:g-frontmutationsign} and \eqref{eq:crwc} are dual equalities, respectively.

Using Proposition \ref{pr:cgrs}, we prove the conjecture \cite[Conjecture~6.10]{fziv}, which is the relation between the $H$-matrices and the $G$-matrices as follows:
\begin{Theorem}[{\cite[Conjecture 6.10]{fziv}}]\label{thm:H=G}
For any $t\in\TT_n$, we have the following relation:
\begin{align} \label{eq:H=G}
H^{B; t_0}_t=-\big[{-}G^{B; t_0}_t\big]_+.
\end{align}
\end{Theorem}

\begin{proof}We assume the sign-coherence of the $C$-matrices. Then, comparing \eqref{eq:grear} with \eqref{eq:grs} and setting $\varepsilon=1$, we get
\begin{gather}\label{eq:BG=BH}
	B\big[{-}G^{B; t_0}_t\big]^{k\bullet}_+=-B\big(H^{B; t_0}_t\big)^{k\bullet}.
\end{gather}
Since $k$ is arbitrary, we have
 \begin{gather}\label{eq:allBG=BH}
	B\big[{-}G^{B; t_0}_t\big]_+=-BH^{B; t_0}_t.
\end{gather}
If $B$ have no zero column vector, then choosing $i$ which satisfies $b_{ik}\neq 0$, we have $[-g_{kj;t}]_+=-h_{kj;t}$ and thus we have \eqref{eq:H=G} as desired. We prove the case that $B$ have $m(\neq0)$ zero column vectors. Permuting labels of $n$-regular tree $\TT_n$, we can assume
\begin{gather*}
B=\begin{bmatrix}
	B' & O \\
	O & O
\end{bmatrix},
\end{gather*}
where $B'$ is $(n-m)\times (n-m)$ matrix without zero column vector. Under this assumption, for \begin{xy}(0,1)*+{t_0}="A",(13.5,1)*+{\cdots}="B",(26.5,1)*+{t}="C",\ar@{-}^{i_1} "A";"B",\ar@{-}^{\ \ i_{s}} "B";"C" \end{xy} in $\TT_n$, we have
\begin{gather*}
\mu_{i_{s}}\cdots\mu_{i_2}\mu_{i_1}\big(G^{B;t_0}_{t_0}\big)|_{n-m} =\begin{bmatrix}\mu_{i_{\ell'}}\cdots\mu_{i'_{2}}\mu_{i'_{1}}\big(G_{t_0}^{B';t_0}\big) \\ O
\end{bmatrix},\\
\mu_{i_{s}}\cdots\mu_{i_2}\mu_{i_1}\big(H^{B;t_0}_{t_0}\big)|_{n-m} =\begin{bmatrix}\mu_{i'_{\ell}}\cdots\mu_{i'_{2}}\mu_{i'_{1}}\big(H_{t_0}^{B';t_0}\big) \\ O
\end{bmatrix},
\end{gather*}
where $(i'_{1},\dots,i'_{\ell})$ is a sequence which is obtained by removing $n-m+1,\dots,n $ from $(i_{1},\dots,i_{s})$, and $|_{n-m}$ means taking the left $n\times(n-m)$ submatrix. Also about the other diagonal entries of $B$, we have the similar equalities. Thus we have
\begin{gather*}
G_t^{B;t_0} =G_{t'}^{B';t_0}\oplus G_{t_1}^{(0);t_0}\oplus\cdots\oplus G_{t_m}^{(0);t_0},\\
H_t^{B;t_0} =H_{t'}^{B';t_0}\oplus H_{t_1}^{(0);t_0}\oplus\cdots\oplus H_{t_m}^{(0);t_0},
\end{gather*}
where $t'$ is a vertex of $\TT_{m-n}$ which satisfies $\Sigma_{t'}=\mu_{i'_{\ell}}\cdots\mu_{i'_{1}}(\Sigma_{t_0})$ and
\begin{gather*}
t_j=\begin{cases}
t_0 &\text{if the number of $n-m+j$ in $(i_{1},\dots,i_{s})$ is even},\\
t'_0 &\text{if the number of $n-m+j$ in $(i_{1},\dots,i_{s})$ is odd},
\end{cases}
\end{gather*}
in $\TT_1$: \begin{xy}(0,1)*+{t_0}="A",(12,1)*+{t'_0}="B",\ar@{-}^{j} "A";"B" \end{xy} for any $j\in\{1,\dots,m\}$.
Explicitly, we have
\begin{gather*}
G^{B;t_0}_{t} =
\begin{bmatrix}
	G^{B';t_0}_{t'} & & &O \\
	 & (-1)^{N_1} & & \\
	 & & \ddots & \\
	 O& & & (-1)^{N_m}
\end{bmatrix},\\
H^{B;t_0}_{t} =
\begin{bmatrix}
	H^{B';t_0}_{t'} & & &O \\
	 & -[(-1)^{N_1+1}]_+ & & \\
	 & & \ddots & \\
	 O& & & -[(-1)^{N_m+1}]_+
\end{bmatrix},
\end{gather*}
where $N_j$ is the number of $n-m+j$ in $(i_{1},\dots,i_{s})$. Since $B'$ has no zero column vectors and
\begin{gather*}
\begin{bmatrix}B'\big[{-}G_{t'}^{B';t_0}\big]_+&O\\O&O\end{bmatrix}=\begin{bmatrix}-B'H_{t'}^{B';t_0}&O\\O&O\end{bmatrix}
\end{gather*}
holds by \eqref{eq:allBG=BH}, we have $\big[{-}G_{t'}^{B';t_0}\big]_+=-H_{t'}^{B';t_0}$. By direct calculation, we also have $\big[{-}G_{t_{j}}^{(0);t_0}\big]_+\allowbreak =-H_{t_j}^{(0);t_0}$ for all $j$. Therefore, we have \eqref{eq:H=G} as desired.
\end{proof}

We point out the following equivalence.
\begin{Proposition} \label{th:cgsign}The following statements are equivalent:
\begin{enumerate}\itemsep=0pt
\item[$(i)$] The $C$-matrices have the column sign-coherence.
\item[$(ii)$] The $G$-matrices have the row sign-coherence, and equality \eqref{eq:H=G} holds.
\end{enumerate}
\end{Proposition}
\begin{proof}We proved (i) $\Rightarrow$ (ii) by \eqref{eq:CG-opponent} and Theorem~\ref{thm:H=G}. We give the proof of~(ii) $\Rightarrow$~(i).
To prove it, let us show~\eqref{eq:CG-opponent}. We note that we can not use \eqref{eq:CG-opponent} directly because we do not assume the sign-coherence of the $C$-matrices now. By assumption, we have $H_t(\varepsilon)^{k\bullet}=-\big[{-}\varepsilon G^{B; t_0}_t\big]^{k\bullet}_+$. Substituting it for~\eqref{eq:crear} and~\eqref{eq:grear}, we have~\eqref{eq:crs} and~\eqref{eq:grs}.
Let us show \eqref{eq:CG-opponent} by the induction on the distance between $t$ and $t_0$. If $\big(G^{B:t_0}_t\big)^T=C^{B^T_t; t}_{t_0}$ holds for some $t \in \TT_n$, then for $t'\in \TT_n$ such that \begin{xy}(0,1)*+{t}="A",(10,1)*+{t'}="B",\ar@{-}^{\ell}"A";"B" \end{xy}, using~\eqref{eq:crs} and~\eqref{eq:g-frontmutation}, we have
\begin{align*}
C^{B^T_{t'}; t'}_{t_0}&=\big(J_{\ell}+\big[{-}\varepsilon B_t^T\big]^{\ell\bullet}_+\big)C^{B_t^T; t}_{t_0}-\big[{-}\varepsilon G^{B_t^T; t}_{t_0}\big]^{\ell\bullet}_+ B^T \\
&=\big(G^{B; t_0}_t\big(J_{\ell}+[-\varepsilon B_t]^{\bullet \ell}_+\big)-B\big[{-}\varepsilon C^{B; t_0}_t\big]^{\bullet \ell}_+\big)^T =\big(G^{B; t_0}_{t'}\big)^T.
\end{align*}
Thus by the row sign-coherence of the $G$-matrices, we get the sign-coherence of the $C$-matrices immediately.
\end{proof}

By Theorem \ref{thm:signs-ci} and Proposition \ref{th:cgsign}, the (row) sign-coherence the $G$-matrices and Theorem~\ref{thm:H=G} hold. Furthermore, applying the sign-coherence of the $C$-matrices is equivalent to applying the sign-coherence of the $G$-matrices and \eqref{eq:H=G}.
 Using it, let us give a reduced expression of the initial-seed mutations of $F$-matrices.
\begin{Proposition}\label{F-rearsigncoherence}\quad
\begin{enumerate}\itemsep=0pt
\item[$(1)$] Let \begin{xy}(0,1)*+{t_0}="A",(10,1)*+{t_1}="B",\ar@{-}^{k}"A";"B" \end{xy} in $\TT_n$, $\mu_k(B)=B_1$ and $\varepsilon \in \{\pm1\}$.
Then, we have
\begin{gather} \label{eq:frs}
F^{B_1; t_1}_t=\big(J_k+[\varepsilon B]^{k\bullet}_+\big)F^{B; t_0}_t +\big[{-}\varepsilon G^{-B; t_0}_t\big]^{k\bullet}_+ +\big[\varepsilon G^{B; t_0}_t\big]^{k\bullet}_+.
\end{gather}
\item[$(2)$] We have a reduced form of the initial-seed mutations as follows:
\begin{align}
F^{B_1; t_1}_t&=\big(J_k+\big[\varepsilon_k\big(G^{-B; t_0}_t\big)B\big]^{k\bullet}_+\big)F^{B; t_0}_t+\big[\varepsilon_k\big(G^{-B; t_0}_t\big)G^{B; t_0}_t\big]^{k\bullet}_+\notag \\
&=\big(J_k+\big[{-}\varepsilon_k\big(G^{B; t_0}_t\big)B\big]^{k\bullet}_+\big)F^{B; t_0}_t+\big[\varepsilon_k\big(G^{B; t_0}_t\big)G^{-B; t_0}_t\big]^{k\bullet}_+.\label{eq:frsred}
\end{align}
\end{enumerate}
\end{Proposition}

\begin{proof}(1) Thanks to Theorem~\ref{thm:H=G}, we can substitute $H^{B; t_0}_t\!(\varepsilon)^{k\bullet}\!=\!{-}\big[{-}\varepsilon G^{B; t_0}_t\big]^{k\bullet}_+\!$ and $H^{-B; t_0}_t\!(\varepsilon)^{k\bullet}\!\allowbreak =-\big[{-}\varepsilon G^{-B; t_0}_t\big]^{k\bullet}_+$ for~\eqref{eq:frear}. Then, we have
\begin{align*}
	F^{B_1; t_1}_t&=\big(J_k+[\varepsilon B]^{k\bullet}_+\big)F^{B; t_0}_t +\big(\varepsilon G^{B; t_0}_t\big)^{k\bullet}+\big[{-}\varepsilon G^{-B; t_0}_t\big]^{k\bullet}_+ +\big[{-}\varepsilon G^{B; t_0}_t\big]^{k\bullet}_+ \\
&=\big(J_k+[\varepsilon B]^{k\bullet}_+\big)F^{B; t_0}_t +\big[{-}\varepsilon G^{-B; t_0}_t\big]^{k\bullet}_+ +\big[\varepsilon G^{B; t_0}_t\big]^{k\bullet}_+,
\end{align*}
as desired.

(2) Substituting $\varepsilon=\varepsilon_k\big(G^{-B; t_0}_t\big)$ or $\varepsilon=-\varepsilon_k\big(G^{B; t_0}_t\big)$ for~\eqref{eq:frs}, we obtain~\eqref{eq:frsred}.
\end{proof}

Like the duality between the final-seed and initial-seed mutations of the $C$- and $G$-matrices, \eqref{eq:ffr} and \eqref{eq:frs}, \eqref{eq:ffront} and \eqref{eq:frsred} are dual equalities, respectively. We show the self-duality of the $F$-matrices, which is analogous to the duality \eqref{eq:CG-opponent} between the $C$- and $G$-matrices. This is the main theorem in this paper.
\begin{Theorem}\label{thm:F-opposite} For any exchange matrix~$B$ and $t_0, t \in \TT_n$, we have
\begin{gather}\label{eq:F-opposite}
\big(F_t^{B;t_0}\big)^T = F_{t_0}^{B_t^T;t}.
\end{gather}
\end{Theorem}
\begin{proof}We prove \eqref{eq:F-opposite} by the induction on the distance between $t$ and $t_0$ in $\TT_n$. When $t=t_0$, we have $\big(F_t^{B;t_0}\big)^T = O =F_{t_0}^{B_t^T;t}$ as desired. We show that if~\eqref{eq:F-opposite} holds for some $t\in\TT_n$, then it also holds for $t'\in\TT_n$ such that \begin{xy}(0,1)*+{t}="A",(10,1)*+{t'}="B",\ar@{-}^{\ell} "A";"B" \end{xy}. By the inductive assumption, \eqref{eq:ffront}, Proposition~\ref{F-rearsigncoherence} and~\eqref{eq:CG-opponent}, we have
\begin{align*}
\big(F_{t'}^{B;t_0}\big)^T&=\big(J_{\ell}+[\varepsilon_{\ell}\big(G^{-B^T_t; t}_{t_0}\big)B_t^T]^{\ell\bullet}_+\big)\big(F^{B; t_0}_t\big)^T+\big[\varepsilon_{\ell}\big(G^{-B_t^T; t}_{t_0}\big)G^{B_t^T; t}_{t_0}\big]^{\ell\bullet}_+ \\
&=\big(J_{\ell}+[\varepsilon_{\ell}\big(G^{-B_t^T; t}_{t_0}\big)B_t^T]^{\ell\bullet}_+\big)F^{B^T_t; t}_{t_0}+\big[\varepsilon_{\ell}\big(G^{-B_t^T; t}_{t_0}\big)G^{B_t^T; t}_{t_0}\big]^{\ell\bullet}_+ =F_{t_0}^{B_{t'}^T;t'}
\end{align*}
as desired.
\end{proof}

\subsection{Examples}\label{section3.4}
We introduce an example for the final-seed and initial-seed mutations in the case of $A_2$. Let $n=2$, and consider a tree $\TT_2$ whose edges are labeled as follows:
\[
\begin{xy}
(-10,1)*+{\dots}="a",(0,1)*+{t_0}="A",(10,1)*+{t_1}="B",(20,1)*+{t_2}="C", (30,1)*+{t_3}="D",(40,1)*+{t_4}="E",(50,1)*+{t_5}="F", (60,1)*+{\dots}="f"
\ar@{-}^{1}"a";"A"
\ar@{-}^{2}"A";"B"
\ar@{-}^{1}"B";"C"
\ar@{-}^{2}"C";"D"
\ar@{-}^{1}"D";"E"
\ar@{-}^{2}"E";"F"
\ar@{-}^{1}"F";"f"
\end{xy}.
\]
We set $B=\left[\begin{smallmatrix}
	0 & 1 \\
	-1 & 0
\end{smallmatrix}\right]$ as the initial exchange matrix at $t_0$. Then, the coefficients, the cluster variables, the $C$-, $G$- and $F$-matrices are given by Tables~\ref{table1} and~\ref{table2} \cite[Example~2.10]{fziv}.
\begin{table}[ht]
\centering
\begin{tabular}{|c|cc|cc|}
\hline
$t$& \multicolumn{2}{|c}{$\Ycal^{B; t_0}_t$} & \multicolumn{2}{|c|}{$\Xcal^{B; t_0}_t$}\tsep{3pt}\bsep{2pt}\\
\hline
&&&&\\[-4mm]
$0$ &$y_1$ & $y_2$& $ x_1$& $x_2$ \\[1mm]
\hline
&&&&\\[-4mm]
$1$& $y_1(y_2\oplus 1)$& $\dfrac{1}{y_2}$ &$ x_1$& $\dfrac{x_1y_2+1}{(y_2\oplus 1)x_2}$ \\[3mm]
\hline
&&&&\\[-4mm]
$2$& $\dfrac{1}{y_1(y_2\oplus 1)}$ & $\dfrac{y_1y_2\oplus y_1\oplus 1}{y_2}$ & $\dfrac{x_1y_1y_2 + y_1+ x_2}{(y_1y_2\oplus y_1\oplus 1)x_1x_2}$ & $\dfrac{x_1y_2+1}{(y_2\oplus 1)x_2}$ \\[3mm]
\hline
&&&&\\[-4mm]
$3$& $\dfrac{y_1\oplus1}{y_1y_2}$ & $\dfrac{y_2}{y_1y_2\oplus y_1\oplus 1} $& $\dfrac{x_1y_1y_2+y_1+x_2}{(y_1y_2\oplus y_1\oplus 1)x_1x_2}$ & $\dfrac{y_1+x_2}{x_1(y_1\oplus 1)}$ \\[3mm]
\hline
&&&&\\[-4mm]
$4$& $\dfrac{y_1y_2}{y_1\oplus 1}$&$\dfrac{1}{y_1}$ &$ x_2 $& $\dfrac{y_1+x_2}{x_1(y_1\oplus 1)}$ \\[3mm]
\hline
&&&&\\[-4mm]
$5$&$ y_2$ &$ y_1$ & $x_2$ & $x_1$\\[1mm]
\hline
\end{tabular}
\caption{Coefficients and cluster variables in type~$A_2$.}\label{table1}
\end{table}

\begin{table}[h]
\centering
\begin{tabular}{|c|c|c|c|}
\hline
&&&\\[-4mm]
$t$& $C^{B; t_0}_t$ & $G^{B; t_0}_t$ \hspace{0mm} & $F^{B;t_0}_t$ \bsep{2pt}\\
\hline
&&&\\[-3mm]
$0$& $\begin{bmatrix}
	1 & 0 \\
	0 & 1
\end{bmatrix} $
& $\begin{bmatrix}
	1 & 0 \\
	0 & 1
\end{bmatrix}$
&$\begin{bmatrix}
	0 & 0 \\
	0 & 0
\end{bmatrix}$
 \\[4mm]
\hline
&&&\\[-3mm]
$1$& $\begin{bmatrix}
	1 & 0 \\
	0 & -1
\end{bmatrix} $
&$\begin{bmatrix}
	1 & 0 \\
	0 & -1
\end{bmatrix}$
&$\begin{bmatrix}
	0 & 0 \\
	0 & 1
\end{bmatrix}$\\[4mm]
\hline
&&&\\[-3mm]
$2$& $\begin{bmatrix}
	-1 & 0 \\
	0 & -1
\end{bmatrix} $
&$\begin{bmatrix}
	-1 & 0 \\
	 0 & -1
\end{bmatrix}$
&$\begin{bmatrix}
	1 & 0 \\
	1 & 1
\end{bmatrix}$\\[4mm]
\hline
&&&\\[-3mm]
3& $\begin{bmatrix}
	-1 & 0 \\
	 -1& 1
\end{bmatrix}$
&$\begin{bmatrix}
	-1 & -1 \\
	0 & 1
\end{bmatrix}$
&$\begin{bmatrix}
	1 & 1 \\
	 1& 0
\end{bmatrix}$\\[4mm]
\hline
&&&\\[-3mm]
$4$& $\begin{bmatrix}
	1 & -1 \\
	1 & 0
\end{bmatrix}$
&$\begin{bmatrix}
	0 & -1 \\
	1 & 1
\end{bmatrix}$
&$\begin{bmatrix}
	0 & 1 \\
	0 & 0
\end{bmatrix}$ \\[4mm]
\hline
&&&\\[-3mm]
$5$& $\begin{bmatrix}
	0 & 1 \\
	1 & 0
\end{bmatrix}$
&$\begin{bmatrix}
	0 & 1 \\
	1 & 0
\end{bmatrix}$
&$\begin{bmatrix}
	0 & 0 \\
	0 & 0
\end{bmatrix}$\\[4mm]
\hline
\end{tabular}
\caption{$C$-, $G$- and $F$-matrices in type~$A_2$.}\label{table2}
\end{table}

We show the expressions of the coefficients and the cluster variables at $t_0$ in Type $A_2$ in Table~\ref{table3}, and its counterpart the $C$-, $G$- and $F$-matrices in Table~\ref{table4}.
\begin{table}[h]
\centering
\begin{tabular}{|c|cc|cc|}
\hline
$t$& \multicolumn{2}{|c}{$\Ycal^{B^T_t; t}_{t_0}$}& \multicolumn{2}{|c|}{$\Xcal^{B^T_t; t}_{t_0}$}\tsep{5pt}\bsep{2pt}\\
\hline
&&&&\\[-4mm]
$0$ &$y_1$ & $y_2$& $x_1$& $x_2$ \\[1mm]
\hline
&&&&\\[-4mm]
$1$& $y_1(y_2\oplus 1)$ & $\dfrac{1}{y_2}$ & $x_1$& $\dfrac{y_2x_1+1}{(y_2\oplus 1)x_2}$ \\[3mm]
\hline
&&&&\\[-4mm]
$2$& $\dfrac{y_1y_2\oplus y_2\oplus 1}{y_1}$ & $\dfrac{1}{y_2( y_1\oplus 1)}$ & $\dfrac{y_1x_2 + 1}{(y_1\oplus 1)x_1}$ & $\dfrac{y_1y_2x_2+y_2+x_1}{(y_1y_2\oplus y_2\oplus 1)x_1x_2}$ \\[3mm]
\hline
&&&&\\[-4mm]
$3$& $\dfrac{y_1\oplus 1}{y_1y_2}$ & $\dfrac{y_2}{y_1y_2\oplus y_1\oplus 1}$ & $\dfrac{y_1y_2x_1+y_1+x_2}{(y_1y_2\oplus y_1\oplus 1)x_1x_2}$ & $\dfrac{y_1+x_2}{(y_1\oplus 1)x_1}$ \\[3mm]
\hline
&&&&\\[-4mm]
$4$& $\dfrac{1}{y_2}$&$\dfrac{y_1y_2}{y_2\oplus 1}$ &$ \dfrac{y_2+x_1}{(y_2\oplus 1)x_2} $& $x_1$ \\[3mm]
\hline
&&&&\\[-4mm]
$5$&$ y_2$ &$y_1$ & $x_2$ &$ x_1$\\[1mm]
\hline
\end{tabular}
\caption{Expressions of coefficients and cluster variables at $t_0$ in type~$A_2$.}\label{table3}
\end{table}

\begin{table}[h]
\centering
\begin{tabular}{|c|c|c|c|}
\hline
&&&\\[-4mm]
$t$& $C^{B^T_t; t}_{t_0}$ &$ G^{B^T_t; t}_{t_0}$ \hspace{0mm} & $F^{B^T_t;t}_{t_0}$ \bsep{2pt}\\
\hline
&&&\\[-3mm]
$0$& $\begin{bmatrix}
	1 & 0 \\
	0 & 1
\end{bmatrix} $
& $\begin{bmatrix}
	1 & 0 \\
	0 & 1
\end{bmatrix}$
&$\begin{bmatrix}
	0 & 0 \\
	0 & 0
\end{bmatrix}$
 \\[4mm]
\hline
&&&\\[-3mm]
$1$& $\begin{bmatrix}
	1 & 0 \\
	0 & -1
\end{bmatrix}$
&$\begin{bmatrix}
	1 & 0 \\
	0 & -1
\end{bmatrix}$
&$\begin{bmatrix}
	0 & 0 \\
	0 & 1
\end{bmatrix}$\\[4mm]
\hline
&&&\\[-3mm]
$2$& $\begin{bmatrix}
	-1 & 0 \\
	0 & -1
\end{bmatrix} $
&$\begin{bmatrix}
	-1 & 0 \\
	 0 & -1
\end{bmatrix}$
&$\begin{bmatrix}
	1 & 1 \\
	0 & 1
\end{bmatrix}$\\[4mm]
\hline
&&&\\[-3mm]
$3$& $\begin{bmatrix}
	-1 & 0 \\
	 -1& 1
\end{bmatrix}$
&$\begin{bmatrix}
	-1 & -1 \\
	0 & 1
\end{bmatrix}$
&$\begin{bmatrix}
	1 & 1 \\
	 1& 0
\end{bmatrix}$\\[4mm]
\hline
&&&\\[-3mm]
$4$& $\begin{bmatrix}
	0 & 1 \\
	-1 & 1
\end{bmatrix}$
&$\begin{bmatrix}
	1 & 1 \\
	-1 & 0
\end{bmatrix}$
&$\begin{bmatrix}
	0 & 0 \\
	1 & 0
\end{bmatrix}$ \\[4mm]
\hline
&&&\\[-3mm]
$5$& $\begin{bmatrix}
	0 & 1 \\
	1 & 0
\end{bmatrix}$
&$\begin{bmatrix}
	0 & 1 \\
	1 & 0
\end{bmatrix}$
&$\begin{bmatrix}
	0 & 0 \\
	0 & 0
\end{bmatrix}$\\[4mm]
\hline
\end{tabular}
\caption{$C$-, $G$- and $F$-matrices in type~$A_2$ (moving the initial vertex).}\label{table4}
\end{table}

Comparing Table~\ref{table2} with Table~\ref{table4}, we can see the duality of the $C$-, $G$- and $F$-matrices in~\eqref{eq:CG-opponent} and~\eqref{eq:F-opposite}.

\section{Properties on principal extension of exchange matrices}\label{section4}

In this section, we present some properties of the principal extension of the exchange matrices. These properties yield alternative derivations of some equalities in the previous sections and are interesting for their own right.
 \subsection{Principal extension}\label{section4.1}
For an exchange matrix $B_t$, $\tilde{B_t}=\left[\begin{smallmatrix} B_t \\ C_t^{B;t_0}\end{smallmatrix}\right]$ $(t\in\TT_n)$ is the extended exchange matrix \cite[Section~5]{fzi}. We fix the initial vertex $t_0\in\TT_{2n}$. Let us regard $\TT_n$ as the subtree of $\TT_{2n}$ induced by vertices which are reachable from~$t_0$ via edges labeled by $1,\dots,n$. Let us consider the following ``full extension of $B$''. We set
\begin{gather}\label{eq:extensionB_0}
\overline{B}:=
\begin{bmatrix}
B & -I_n \\
I_n & O
\end{bmatrix}.
\end{gather}
Note that $\overline{B}$ is regular. Then, we obtain a family of $2n\times 2n$ skew-symmetrizable matrices $\{\overline{B}_t\}_{t\in\TT_{2n}}$ such that $\overline{B}=\overline{B}_{t_0}$ $(t_0\in\TT_n\subset \TT_{2n})$ and they are related by the mutation~\eqref{eq:matrix-mutation}. Then, the left half of $\overline{B}_t$ is $\tilde{B}_t$ for $t\in\TT_n$, and also $\overline{B}_t$ is regular. We call $\{\overline{B}_t\}_{t\in\TT_{2n}}$ the \emph{principal extension} of $\{B_t\}_{t\in\TT_n}$. The following proposition gives the explicit expression of $\overline{B}_t$ $(t\in\TT_n)$:
\begin{Proposition}\label{lem:extension B_t} Let $\{\Sigma_t=(\xx_t, \yy_t, B_t)\}_{t\in\TT_{n}}$ be a cluster pattern with the initial vertex $t_0$. For any mutation series $\mu_{i_m}\mu_{i_{m-1}}\cdots\mu_{i_1}\ (i_1,\dots,i_m\in\{1,\dots,n\})$ of $\TT_{n}$, we set $t_{1},\dots,t_{{m-1}},t_{m},t$ as \begin{xy}(0,1)*+{t_0}="A",(13,1)*+{t_1}="B",(27,1)*+{\cdots}="C",(43,1)*+{t_{m-1}}="D",(62,1)*+{t_m=t}="E",\ar@{-}^{i_1} "A";"B",\ar@{-}^{i_2} "B";"C",\ar@{-}^{ i_{m-1}} "C";"D",\ar@{-}^{ i_{m}} "D";"E" \end{xy} in $ \TT_{n}$. Then, the following equality holds:
\begin{gather}\label{eq:extension B_tsign}
\overline{B}_t=
\begin{bmatrix}
B_t & -D^{-1}\big(C_t^{B;t_0}\big)^TD\\
C_t^{B;t_0} & O
\end{bmatrix},
\end{gather}
where $D$ is a skew-symmetrizer of $B$.
\end{Proposition}
\begin{proof}We start with the proof of the following equality:
\begin{gather}\label{eq:extension B_t}
\overline{B}_t=
\begin{bmatrix}
B_t & -D^{-1}\big(C_t^{B;t_0}\big)^TD\\
C_t^{B;t_0} &\mathop{\sum}\limits_{s=0}^{m-1} \big(C_{t_s}^{B;t_0}\big[{-}D^{-1}\big(C_{t_s}^{B;t_0}\big)^TD\big]_+^{i_{s+1}\bullet}-\big[{-}C_{t_s}^{B;t_0}\big]_+^{\bullet i_{s+1}}D^{-1}\big(C_{t_s}^{B;t_0}\big)^TD\big)
\end{bmatrix}.
\end{gather}
We prove \eqref{eq:extension B_t} by the induction on the distance between $t$ and $t_0$ in $\TT_{n}$. The base case $m=0$ is immediate as $\overline{B_{t_0}}=\overline{B}$. It remains to show that if~\eqref{eq:extension B_t} holds for some $t \in \TT_{n}$, then it also holds
for $t' \in \TT_n$ such that\begin{xy}(0,1)*+{t}="A",(10,1)*+{t'}="B",\ar@{-}^\ell "A";"B" \end{xy}. The left $2n \times n$ submatrix of $\overline{B}_t$ coincides with the extended matrix $\tilde{B}_t$ in \cite[Section 5]{fzi}. Therefore, it suffices to prove the right $2n \times n$ submatrix of $\overline{B}_t$. For any $1\leq i \leq n$ and $n+1\leq j \leq 2n$, we have
\begin{align*}
b_{ij;t'}&=\begin{cases}-b_{ij;t} &\text{if $i=\ell$,} \\
b_{ij;t}+b_{i\ell;t} [ b_{\ell j;t}] _{+}+ [ -b_{i\ell;t}] _{+}b_{\ell j;t} &\text{otherwise}
\end{cases}\\
&=\begin{cases}d_i^{-1}c_{j-n,i;t}d_j &\text{if $i=\ell$,}\\
-d_i^{-1}c_{j-n,i;t}d_j+b_{i\ell;t}\big[{-}d_\ell^{-1}c_{j-n,\ell;t}d_j\big]_{+}- [-b_{i\ell;t}]_+d_\ell^{-1}c_{j-n,\ell;t}d_j &\text{otherwise}
\end{cases}\\
&=\begin{cases}-d_i^{-1}c_{j-n,i;t'}d_j &\text{if $i=\ell$,} \\
-d_i^{-1}(c_{j-n,i;t}+b_{\ell i;t} [-c_{j-n,\ell;t} ]_{+}+ [b_{\ell i;t}]_+c_{j-n,\ell;t})d_j &\text{otherwise}
\end{cases}\\
&=-d_i^{-1}c_{j-n,i;t'}d_j .
\end{align*}
By using the equality
\begin{gather*}
\big(C_t^{B;t_0}\big[{-}D^{-1}\big(C_t^{B;t_0}\big)^TD\big]_+^{\ell\bullet}-\big[{-}C_t^{B;t_0}\big]_+^{\bullet\ell}D^{-1}\big(C_t^{B;t_0}\big)^TD\big)_{i-n,j-n}\\
\qquad {} =c_{i-n,\ell;t}\big[{-}d_{\ell}^{-1}c_{j-n,\ell;t}d_j\big]_+-[-c_{i-n,\ell;t}]_+d_{\ell}^{-1}c_{j-n,\ell;t}d_j
\end{gather*}
and the inductive assumption, for any $n+1\leq i,j \leq 2n$, we have
\begin{align}
b_{ij;t'}& = b_{ij;t}+b_{i\ell} [ b_{\ell j;t} ] _{+}+ [ -b_{i\ell;t} ] _{+}b_{\ell j;t}\nonumber\\
&=\sum_{s=0}^{{m-1}}\big(c_{i-n,{i_{s+1}};t_s}\big[{-}d_{i_{s+1}}^{-1}c_{j-n,{i_{s+1}};t_s}d_j\big]_+-[-c_{i-n,{i_{s+1}};t_s}]_+d_{i_{s+1}}^{-1}c_{j-n,{i_{s+1}};t_s}d_j\big)\nonumber\\
&\quad{} +c_{i-n,\ell;t}\big[{-}d_{\ell}^{-1}c_{j-n,\ell;t}d_j\big]_+-[-c_{i-n,\ell;t}]_+d_{\ell}^{-1}c_{j-n,\ell;t}d_j\nonumber\\
&=\sum_{s=0}^{m}\big(c_{i-n,{i_{s+1}};t_s}\big[{-}d_{i_{s+1}}^{-1}c_{j-n,{i_{s+1}};t_s}d_j\big]_+-[-c_{i-n,{i_{s+1}};t_s}]_+d_{i_{s+1}}^{-1}c_{j-n,{i_{s+1}};t_s}d_j\big) \nonumber\\
&=\sum_{s=0}^{m} \big(C_{t_s}^{B;t_0}\big[{-}D^{-1}\big(C_{t_s}^{B;t_0}\big)^TD\big]_+^{{i_{s+1}}\bullet}-\big[{-}C_{t_s}^{B;t_0}\big]_+^{\bullet{i_{s+1}}}D^{-1}\big(C_{t_s}^{B;t_0}\big)^TD\big)_{i-n,j-n}, \label{eq:b=dcd}
\end{align}
where $t_m=t$ and $\ i_{m+1}=\ell$. Thus we have \eqref{eq:extension B_t}. Moreover, applying the sign-coherence of the $C$-matrices, the third equality of~\eqref{eq:b=dcd} is always zero. Therefore, we have~\eqref{eq:extension B_tsign}.
\end{proof}

Let us study the $C$-, $G$- and $F$-matrices for the initial matrix $\overline{B}$ in \eqref{eq:extensionB_0}, especially for the subtree $\TT_n$ in $\TT_{2n}$.
\begin{Theorem}{\label{thm:CGF-correspondence}}
Let $\{\Sigma_t=(\xx_t, \yy_t, B_t)\}_{t\in\TT_{n}}$ be a cluster pattern with the initial vertex $t_0$. For any mutation series $\mu_{i_m}\mu_{i_{m-1}}\cdots\mu_{i_1}\ (i_1,\dots,i_m\in\{1,\dots,n\})$ of $\TT_{n}$, we set $t_{1},\dots,t_{{m-1}},t_{m},t$ as \begin{xy}(0,1)*+{t_0}="A",(13,1)*+{t_1}="B",(27,1)*+{\cdots}="C",(43,1)*+{t_{m-1}}="D",(62,1)*+{t_m=t}="E",\ar@{-}^{i_1} "A";"B",\ar@{-}^{i_2} "B";"C",\ar@{-}^{ i_{m-1}} "C";"D",\ar@{-}^{ i_{m}} "D";"E" \end{xy} in $ \TT_{n}$. Then, the following equalities hold:
\begin{gather}\label{eq:C-correspondencesign}
C_t^{\overline{B};t_0}=\begin{bmatrix}
C_t^{B;t_0}&O\\
O&I_n
\end{bmatrix},\\
\label{eq:G-correspondence}
G_t^{\overline{B};t_0}=\begin{bmatrix}
G_t^{B;t_0}&O\\
O&I_n
\end{bmatrix},\\
\label{eq:F-correspondence}
F_{\ell;t}^{\overline{B};t_0}(\yy)=\begin{cases}
F_{\ell;t}^{B;t_0}(\yy) &\text{if }\ell\in\{1,\dots, n\}, \\
1 &\text{if }\ell\in\{n+1,\dots, 2n\}.
\end{cases}
\end{gather}
\end{Theorem}
\begin{proof}Firstly, we start the proof of the following equality:
\begin{gather}\label{eq:C-correspondence}
C_t^{\overline{B};t_0} =\begin{bmatrix}
C_t^{B;t_0}&\mathop{\sum}\limits_{s=0}^{m-1} \big(C_{t_s}^{B;t_0}\big[D^{-1}\big(C_{t_s}^{B;t_0}\big)^TD\big]_+^{i_{s+1}\bullet}-\big[C_{t_s}^{B;t_0}\big]_+^{\bullet i_{s+1}}D^{-1}\big(C_{t_s}^{B;t_0}\big)^TD\big)\\
O&I_n
\end{bmatrix}.
\end{gather}
We prove \eqref{eq:C-correspondence} by the induction on the distance between $t$ and $t_0$ in $\TT_{n}$. The base case $m=0$ is immediate as $C_{t_0}^{\overline{B};t_0}=I_{2n}$. It remains to show that if \eqref{eq:C-correspondence} holds for some $t \in \TT_n$, then it also holds
for $t' \in \TT_n$ such that\begin{xy}(0,1)*+{t}="A",(10,1)*+{t'}="B",\ar@{-}^\ell "A";"B" \end{xy}.
By \eqref{eq:c-frontmutation}, we have
\begin{align*}
C_{t'}^{\overline{B};t_0}&= C_{t}^{\overline{B};t_0}\big(J_{\ell}+[-\overline{B}_t]_+^{\ell\bullet }\big)+\big[C^{\overline{B};t_0}_t\big]_+^{\bullet \ell}\overline{B}_t\\
 & =\begin{bmatrix}
C_t^{B;t_0}&X\\
O&I_n
\end{bmatrix}
\left(\begin{bmatrix}
J_\ell&O\\
O&I_n
\end{bmatrix}+\begin{bmatrix}
[-B_t]_+^{\ell\bullet}&[D^{-1}\big(C_t^{B;t_0}\big)^TD]_+^{\ell\bullet}\\
O&O
\end{bmatrix}\right)\\
&\quad {}+\begin{bmatrix}
\big[C_t^{B;t_0}\big]_+^{\bullet \ell}&O\\
O&O
\end{bmatrix}
\begin{bmatrix}
B_t&-D^{-1}\big(C_t^{B;t_0}\big)^TD\\
C_t^{B;t_0}&Y
\end{bmatrix}\\
& =\begin{bmatrix}
C_{t}^{B;t_0}\big(J_{\ell}+[-{B}_t]_+^{\ell\bullet }\big)+\big[C^{B;t_0}_t\big]_+^{\bullet \ell}{B}_t&Z\\
O&I_n
\end{bmatrix} =\begin{bmatrix}
C_{t'}^{B;t_0}&Z\\
O&I_n
\end{bmatrix},
\end{align*}
where
\begin{gather*}
X =\mathop{\sum}\limits_{s=0}^{m-1} \big(C_{t_s}^{B;t_0}\big[D^{-1}\big(C_{t_s}^{B;t_0}\big)^TD\big]_+^{i_{s+1}\bullet}-\big[C_{t_s}^{B;t_0}\big]_+^{\bullet i_{s+1}}D^{-1}\big(C_{t_s}^{B;t_0}\big)^TD\big),\\
Y =\mathop{\sum}\limits_{s=0}^{m-1} \big(C_{t_s}^{B;t_0}\big[{-}D^{-1}\big(C_{t_s}^{B;t_0}\big)^TD\big]_+^{i_{s+1}\bullet}-\big[{-}C_{t_s}^{B;t_0}\big]_+^{\bullet i_{s+1}}D^{-1}\big(C_{t_s}^{B;t_0}\big)^TD\big),\\
Z =X+C_t^{B;t_0}\big[D^{-1}\big(C_t^{B;t_0}\big)^TD\big]_+^{\ell\bullet}-\big[C_t^{B;t_0}\big]_+^{\bullet \ell}D^{-1}\big(C_t^{B;t_0}\big)^TD.
\end{gather*}
Calculating $Z$, we have
\begin{gather*}
Z = \sum _{s=0}^{m} \big(C_{t_s}^{B;t_0}\big[D^{-1}\big(C_{t_s}^{B;t_0}\big)^TD\big]_+^{i_{s+1}\bullet}-\big[C_{t_s}^{B;t_0}\big]_+^{\bullet i_{s+1}}D^{-1}\big(C_{t_s}^{B;t_0}\big)^TD\big),
\end{gather*}
where $t_m=t$ and $i_{m+1}=\ell$. Thus we have~\eqref{eq:C-correspondence}. Furthermore, since the $(i, j)$ entry of
\begin{gather*}
C_t^{B;t_0}\big[D^{-1}\big(C_t^{B;t_0}\big)^TD\big]_+^{\ell\bullet}-\big[C_t^{B;t_0}\big]_+^{\bullet \ell}D^{-1}\big(C_t^{B;t_0}\big)^TD
\end{gather*}
is
\begin{gather*}
d^{-1}_{\ell}(c_{i\ell;t}[c_{j\ell;t}]_+-[c_{i\ell;t}]c_{j\ell;t})d_j,
\end{gather*}
we have $Z=O$ by applying the sign-coherence of the $C$-matrices. Hence we have \eqref{eq:C-correspondencesign}.
Secondly, we prove \eqref{eq:G-correspondence} by the same way as \eqref{eq:C-correspondence}. The base case $m=0$ is immediate as $G_{t_0}^{\overline{B};t_0}=I_{2n}$. By the inductive assumption, we have
\begin{align*}
G_{t'}^{\overline{B};t_0}&=G_{t}^{\overline{B};t_0}\big(J_{\ell}+[\overline{B}_t]_+^{\bullet \ell}\big)-\overline{B}\big[C_t^{\overline{B};t_0}\big]_+^{\bullet \ell}\\
&=\begin{bmatrix}
G_t^{B;t_0}&O\\
O&I_n
\end{bmatrix}
\left(\begin{bmatrix}
J_\ell&O\\
O&I_n
\end{bmatrix}+\begin{bmatrix}
[{B}_t]_+^{\bullet \ell}&O\\
\big[C_t^{B;t_0}\big]_+^{\bullet \ell}&O
\end{bmatrix}\right)-\begin{bmatrix}
B&-I_n\\
I_n&O
\end{bmatrix}
\begin{bmatrix}
\big[C_t^{B;t_0}\big]_+^{\bullet \ell}&O\\
O&O
\end{bmatrix}\\
&=\begin{bmatrix}
G_{t}^{{B};t_0}\big(J_{\ell}+[{B}_t]_+^{\bullet \ell}\big)-B\big[C_t^{B;t_0}\big]_+^{\bullet \ell}&O\\
O&I_n
\end{bmatrix} =\begin{bmatrix}
G_{t'}^{B;t_0}&O\\
O&I_n
\end{bmatrix}
\end{align*}
as desired.
Finally, we prove \eqref{eq:F-correspondence} by the same way as \eqref{eq:C-correspondence} or \eqref{eq:G-correspondence}. The base case $m=0$ is immediate as $F_{i;t_0}^{\overline{B};t_0}(\yy)=1 $ for all $i$. Let $C_t^{\overline{B};t_0}=(\overline{c}_{ij;t})$ and $\overline{B}_t=(\overline{b}_{ij;t})$, and we abbreviate $F_{i;t}^{B;t_0}(\yy)=F_{i;t}$ and $F_{i;t}^{\overline{B};t_0}(\yy)=\overline{F}_{i;t}$. By the inductive assumption and Proposition \ref{lem:extension B_t} and \eqref{eq:C-correspondence}, we have
\begin{align*}
\overline{F}_{i;t'}&= \overline{F}_{i;t}=F_{i;t} =F_{i;t'} \qquad \text{if $i\neq \ell$,}\\
\overline{F}_{\ell;t'} &= \frac{\prod\limits_{j=1}^{2n} y_j^{[\overline{c}_{j\ell;t}]_+}\prod\limits_{i=1}^{2n} \overline{F}_{i;t}^{[\overline{b}_{i\ell;t}]_+}
+ \prod\limits_{j=1}^{2n} y_j^{[-\overline{c}_{j\ell;t}]_+} \prod\limits_{i=1}^{2n} \overline{F}_{i;t}^{[-\overline{b}_{i\ell;t}]_+}}{\overline{F}_{\ell;t}}\\
&= \frac{\prod\limits_{j=1}^{n} y_j^{[c_{j\ell;t}]_+} \prod\limits_{i=1}^{n} F_{i;t}^{[b_{i\ell;t}]_+} + \prod\limits_{j=1}^{n} y_j^{[c_{j\ell;t}]_+} \prod\limits_{i=1}^n F_{i;t}^{[-b_{i\ell;t}]_+}}{F_{\ell;t}}
={F}_{\ell;t'}
\end{align*}
as desired.
\end{proof}

\begin{Remark}\label{re:remark of C} The equality \eqref{eq:C-correspondencesign} is the specialisation of \cite[equation~(2)]{cl}. Indeed, setting $m=n$ and substituting $B$, $I_n$, $-I_n$ and $O$ for $B_1$, $B_2$, $B_3$ and $B_4$ in $\tilde{B}=
\left(\begin{smallmatrix}
B_1&B_2\\
B_3&B_4\\
I_n&0\\
0&I_m
\end{smallmatrix}\right)$ appearing in \cite[Proof of Theorem~4.5]{cl}, we have \eqref{eq:C-correspondencesign} by \cite[equation~(2)]{cl}.
\end{Remark}

\begin{Remark}\label{re:remark of F}We did not use the sign-coherence of the $C$-matrices when we show \eqref{eq:G-correspondence} and~\eqref{eq:F-correspondence}. Furthermore, \eqref{eq:F-correspondence} always holds when $\overline{B}$ is a $2n\times 2n$ matrix whose the upper left $n\times n$ submatrix is $B$.
\end{Remark}
By Theorem \ref{thm:CGF-correspondence} and Remark \ref{re:remark of F}, we have the following corollary:
\begin{Corollary}\label{thm:FHmat-correspondence} Let $\{\Sigma_t=(\xx_t, \yy_t, B_t)\}_{t\in\TT_{n}}$ be a cluster pattern with the initial vertex $t_0$. Then, for any $t\in\TT_n$, the following equalities hold:
\begin{gather} \label{eq:Fmat-correspondence}
F_t^{\overline{B};t_0} =\begin{bmatrix}
F_t^{B;t_0}&O\\
O&O
\end{bmatrix},\\
\label{eq:hextend}
H^{\overline{B}; t_0}_t = \begin{bmatrix}
	H^{B; t_0}_t & O \\
	O & O
\end{bmatrix}.
\end{gather}
\end{Corollary}

\begin{proof}We have \eqref{eq:Fmat-correspondence} immediately by \eqref{eq:F-correspondence}. Except for the lower left $n\times n$ submatrix, \eqref{eq:hextend} is obtained from the definition of the $H$-matrices and \eqref{eq:F-correspondence}.
For $i\in \{n+1, \dots, 2n\}$ and $j\in\{1, \dots, n\}$, we have
\begin{align*}
u^{h_{ij;t}}&=F^{\overline{B}; t_0}_{j;t}|_{\text{Trop}(u)}\big(u^{[-\overline{b_{i1}}]_+}, \dots, u^{-1}, \dots, u^{[-\overline{b_{i,2n}}]_+}\big) \qquad \big(\text{$u^{-1}$ in the $i$th position}\big) \\
&= F^{B; t_0}_{j;t}|_{\text{Trop}(u)}\big(u^0, \dots, u^0\big)=1.
\end{align*}
Hence we obtain $h_{ij;t}=0$.
\end{proof}

\subsection{Alternative derivations of Propositions \ref{pr:ffr}, \ref{pr:frear}, Theorems~\ref{thm:H=G} and \ref{thm:F-opposite}}\label{section4.2}
\looseness=-1 Using the principal extension, we give alternative derivations of Propositions~\ref{pr:ffr},~\ref{pr:frear}, Theorems~\ref{thm:H=G} and~\ref{thm:F-opposite}. We think that these derivations are important in point of the following. That is, we cancel out exchange matrix $B_t$ in equations to derivate the desired equation in a process of proving. By the principal extension, we can do it even if $B_t$ is not invertible in various cases. We believe that we also apply this technique to any other proofs of theorems in cluster algebra theory.

\begin{proof}[Alternative derivation of Proposition \ref{pr:ffr}] For any edge \begin{xy}(0,1)*+{t}="A",(10,1)*+{t'}="B",\ar@{-}^{\ell}"A";"B" \end{xy} in $\TT_n$, by using \eqref{eq:g-frontmutation} and~\eqref{eq:G=G+BF}, {we obtain two expressions of $G^{-B; t_0}_{t'}$ by the following diagram:}
\[
	\xymatrix{
		G^{B; t_0}_t\ar[r]^{\eqref{eq:G=G+BF}} \ar[d]_{\text{final-seed mutation \eqref{eq:g-frontmutation}}}& G^{-B; t_0}_t \ar[d]^{\text{final-seed mutation \eqref{eq:g-frontmutation}}}\\
		G^{B; t_0}_{t'} \ar[r]_{\eqref{eq:G=G+BF}} & G^{-B; t_0}_{t'}.
	}
\]
Applying \eqref{eq:G=G+BF} to $G^{B; t_0}_t$ before applying the final-seed mutation, we have
\begin{gather*}
G^{-B; t_0}_{t'}=G_{t'}^{B;t_0}+BF_{t'}^{B;t_0} =G^{B; t_0}_t\big(J_\ell+[\varepsilon B_t]^{\bullet \ell}_+\big)-B\big[\varepsilon C^{B; t_0}_t\big]^{\bullet \ell}_+ +BF_{t'}^{B;t_0}.
\end{gather*}
On the other hand, applying the final-seed mutation to $G^{B; t_0}_t$ before applying \eqref{eq:G=G+BF}, we have
\begin{align*}
G^{-B; t_0}_{t'}&=G^{-B; t_0}_t\big(J_\ell+[\varepsilon (-B_t)]^{\bullet \ell}_+\big)-(-B)\big[\varepsilon C^{-B; t_0}_t\big]^{\bullet \ell}_+\\
&=\big(G_t^{B;t_0}+BF_t^{B;t_0}\big)\big(J_\ell+[-\varepsilon B_t]^{\bullet \ell}_+\big)+B\big[\varepsilon C^{-B; t_0}_t\big]^{\bullet \ell}_+.
\end{align*}
Comparing these two expressions, we obtain
\begin{gather}\label{eq:BF=GB+BFJB+BC+BC}
BF^{B; t_0}_{t'}=G^{B; t_0}_t(-\varepsilon B_t)^{\bullet\ell}+BF^{B; t_0}_t\big(J_{\ell}+[-\varepsilon B_t]^{\bullet\ell}_+\big) +B\big[\varepsilon C^{-B; t_0}_t\big]^{\bullet\ell}_+ +B\big[\varepsilon C^{B; t_0}_t\big]^{\bullet\ell}_+.
\end{gather}
Hence if $B$ is invertible, we have
\begin{gather}\label{eq:F=BGB+FJB+C+C}
	F^{B; t_0}_{t'}=B^{-1}G^{B; t_0}_t(-\varepsilon B_t)^{\bullet\ell}+F^{B; t_0}_t\big(J_{\ell}+[-\varepsilon B_t]^{\bullet\ell}_+\big)+
\big[\varepsilon C^{-B; t_0}_t\big]^{\bullet\ell}_+ +\big[\varepsilon C^{B; t_0}_t\big]^{\bullet\ell}_+.
\end{gather}
Since $B^{-1}G^{B; t_0}_t=C^{B; t_0}_t B_t^{-1}$, the first term of the right hand side of \eqref{eq:F=BGB+FJB+C+C} is $\big({-}\varepsilon C^{B; t_0}_t\big)^{\bullet\ell}$.
Moreover, we note that we have
\begin{gather*}
	\big({-}\varepsilon C^{B; t_0}_t\big)^{\bullet\ell}+\big[\varepsilon C^{B; t_0}_t\big]^{\bullet\ell}_+=\big[{-}\varepsilon C^{B; t_0}_t\big]^{\bullet\ell}_+.
\end{gather*}
Thus we have \eqref{eq:ffr} as desired. When $B$ is not invertible, substituting $\overline{B}$ for $B$ in \eqref{eq:BF=GB+BFJB+BC+BC}, we have the equality which is substituted $\overline{B}$ for $B$ in \eqref{eq:F=BGB+FJB+C+C}. Thanks to \eqref{eq:C-correspondence} and \eqref{eq:Fmat-correspondence}, comparing the upper left $n\times n$ submatrix of it, we have \eqref{eq:ffr} as desired.
\end{proof}

\begin{proof}[Alternative derivation of Proposition \ref{pr:frear}]
We prove \eqref{eq:frear}. By using \eqref{eq:crear} and \eqref{eq:c-}, {we obtain two expressions of $C^{-B_1; t_1}_t$ by the following diagram:}
\begin{align*}
	\xymatrix{
		C^{B; t_0}_t\ar[r]^{\eqref{eq:c-}} \ar[d]_{\text{initial-seed mutation \eqref{eq:crear}}}& C^{-B; t_0}_t \ar[d]^{\text{initial-seed mutation \eqref{eq:crear}}} \\
		C^{B_1; t_1}_t \ar[r]_{\eqref{eq:c-}} & C^{-B_1; t_1}_t.
	}
\end{align*}
Applying \eqref{eq:c-} to $C_t^{B;t_0}$ before applying the initial-seed mutation, we have
\begin{align*}
	C^{-B_1; t_1}_t&=\big(J_k+[-\varepsilon (-B)]^{k\bullet}_+\big)C^{-B; t_0}_t+H^{-B; t_0}_t(\varepsilon)^{k\bullet} (-B_t) \\
&=\big(J_k+[\varepsilon B]^{k\bullet}_+\big)\big(C^{B; t_0}_t+F^{B; t_0}_tB_t\big)-H^{-B; t_0}_t(\varepsilon)^{k\bullet}B_t.
\end{align*}
On the other hand, applying the initial-seed mutation to $C_t^{B;t_0}$ before applying \eqref{eq:c-}, we have
\begin{gather*}
C^{-B_1; t_1}_t =C^{B_1; t_1}_t+F^{B_1; t_1}_tB_t =\big(J_k+[-\varepsilon B]^{k\bullet}_+\big)C^{B; t_0}_t+H^{B; t_0}_t(\varepsilon)^{k\bullet} B_t+F^{B_1; t_1}_tB_t.
\end{gather*}
Comparing these two expressions, we have
\begin{gather*}
F^{B_1; t_1}_tB_t=(\varepsilon B)^{k\bullet}C^{B; t_0}_t+\big(J_k+[\varepsilon B]^{k\bullet}_+\big)F^{B; t_0}_tB_t-H^{-B; t_0}_t(\varepsilon)^{k\bullet}B_t-H^{B; t_0}_t(\varepsilon)^{k\bullet} B_t.
\end{gather*}
When $B$ is invertible, all $B_t$ also is. Therefore, we have
\begin{gather}\label{eq:F=BCB}
F^{B_1; t_1}_t=(\varepsilon B)^{k\bullet}C^{B; t_0}_tB^{-1}_t+\big(J_k+[\varepsilon B]^{k\bullet}_+\big)F^{B; t_0}_t-H^{-B; t_0}_t(\varepsilon)^{k\bullet}-H^{B; t_0}_t(\varepsilon)^{k\bullet}.
\end{gather}
By \eqref{GB=BC}, $BC^{B; t_0}_t=G^{B; t_0}_tB_t$ for any $t$, the first term of right hand side of \eqref{eq:F=BCB} is $\big(\varepsilon G^{B; t_0}_t\big)^{k\bullet}$.
Hence we have desired equality. On the other hand, when $B$ is not invertible, we consider replacing $B$ with $\overline{B}$ in \eqref{eq:F=BCB}.
Thanks to \eqref{eq:G-correspondence} and Corollary \ref{thm:FHmat-correspondence}, we have \eqref{eq:frear}.
\end{proof}

\begin{proof}[Alternative derivation of Theorem~\ref{thm:H=G}] We get \eqref{eq:BG=BH} by the same way as a previous proof. If $B$ is invertible, then we have $-\big[{-}G^{B; t_0}_t\big]^{k\bullet}_+=\big(H^{B; t_0}_t\big)^{k\bullet}$. Applying this equality to all $k=1, \dots, n$, we have $H^{B; t_0}_t=-\big[{-}G^{B; t_0}_t\big]$. If $B$ is not invertible, replacing $B$ with $\overline{B}$ in \eqref{eq:BG=BH}, we have $-\big[{-}G^{\overline{B}; t_0}_t\big]^{k\bullet}_+=\big(H^{\overline{B}; t_0}_t\big)^{k\bullet}$. Comparing the upper left $n\times n$ submatrices of both sides of it, we obtain $-\big[{-}G^{B; t_0}_t\big]^{k\bullet}_+=\big(H^{B; t_0}_t\big)^{k\bullet}$.
\end{proof}

\begin{proof}[Alternative derivation of Theorem \ref{thm:F-opposite}]
By the equalities which substitute $B_t^T$, $t_0$ and $t$ for $B$, $t_0$ and $t$ respectively in \eqref{eq:c-}, \eqref{eq:G=G+BF} and \eqref{eq:CG-opponent}, we have
\begin{gather}\label{eq:BF=BF}
B\big(F_t^{B;t_0}\big)^T = BF_{t_0}^{B_t^T;t}.
\end{gather}
When $B$ is invertible, we get \eqref{eq:F-opposite} immediately. We prove the case that $B$ is not invertible. Replacing $B$ with $\overline{B}$ in \eqref{eq:BF=BF}, we have
\begin{gather*}
\big(F_t^{\overline{B};t_0}\big)^T = F_{t_0}^{\overline{B}_t^T;t}
\end{gather*}
because $\overline{B}$ is invertible. Since we have
\begin{gather*}
\big(F_t^{\overline{B};t_0}\big)^T =\begin{bmatrix}
\big(F_t^{B;t_0}\big)^T&O\\
O&O
\end{bmatrix}
\end{gather*}
and
\begin{gather*}
F_{t_0}^{\overline{B}_t^T;t}=F_{t_0}^{\overline{B^T}_t;t}=
\begin{bmatrix}
F_{t_0}^{B_t^T;t}&O\\
O&O
\end{bmatrix}
\end{gather*}
by \eqref{eq:F-correspondence} and Remark \ref{re:remark of F}, we have \eqref{eq:F-opposite}.
\end{proof}

\subsection*{Acknowledgments}
The authors are grateful to Professor Tomoki Nakanishi for useful comments and advice. We also thanks Toshiya Yurikusa for insightful comments about the principal extension. We also would like express our gratitude to the referees.

\pdfbookmark[1]{References}{ref}
\LastPageEnding

\end{document}